\documentclass[twoside]{amsart}

\usepackage{graphicx}
\usepackage{epsfig}
\usepackage{amssymb}
\usepackage{amsmath}
\usepackage{amsfonts}
\usepackage[latin1]{inputenc}
\usepackage{fancyhdr}
\usepackage{lastpage}
\usepackage{pifont}   % para puntuaci\'on ornamental
\usepackage{titlesec}

\usepackage{amsthm}
\usepackage{epsfig,color}
\usepackage[all]{xy}
\newcommand{\numpagesize}{\fontsize{9}{0.1pt}\normalfont}

\fancypagestyle{plain}{ \fancyhf{}

\fancyfoot[CO]{$\overline{\makebox[15mm]{\numpagesize{\thepage}}}$}
\fancyfoot[CE]{$\overline{\makebox[15mm]{\numpagesize{\thepage}}}$}
}

\titleformat{\section}[hang]
 {\fontsize{12}{11.7} \bfseries}
 {\thesection}
 {3mm}
 {}
\titlespacing{\section}{0pt}{23.4pt}{11.7pt}

\titleformat{\subsection}[hang]
 {\fontsize{12}{11.7} \itshape}
 {\thesubsection}
 {3mm}
 {}
\titlespacing{\subsection}{0pt}{23.4pt}{11.7pt}

\newcommand{\A}{\mathcal{A}}

\newcommand{\K}{\mathbb{K}}

\newcommand{\id}{\mbox{id}}

\newcommand{\sll}{\mathfrak{sl}_2(\mathbb{K})}

\newtheorem{theorem}{Theorem}[section]
\newtheorem{corollary}[theorem]{Corollary}
\newtheorem{lemma}[theorem]{Lemma}
\newtheorem{proposition}[theorem]{Proposition}

\newtheorem{definition}[theorem]{Definition}

\newtheorem{example}[theorem]{Example}

\newtheorem{remark}[theorem]{Remark}

\def\id{{\rm id}}

\begin{document}

\thispagestyle{plain}

\title {Paradigm of Nonassociative \\ Hom-algebras and
Hom-superalgebras}

\author{ Abdenacer Makhlouf}

\address{  Laboratoire de Math\'ematiques, Informatique et
Applications\\ Universit\'e de Haute-Alsace, France}
\email{Abdenacer.Makhlouf@uha.fr}

\vspace*{9.6mm}

\maketitle

%%%%%%
%%%%%%%  AQUI EMPIEZA EL CONTENIDO
%%%%%%%

\section*{Abstract}
The aim of this paper is to give a survey of nonassociative
Hom-algebra and Hom-superalgebra structures. The main feature of
these algebras is that the identities defining the structures are
twisted by  homomorphisms. We
 discuss  Hom-associative algebras, Hom-Flexible algebras, Hom-Lie algebras,
 $G$-hom-associative algebras, Hom-Poisson algebras,
 Hom-alternative algebras and Hom-Jordan algebras and  $\mathbb{Z}_2$-graded versions.
  We give an overview of the development of Hom-algebras structures
which has been intensively investigated recently.

\medskip

{\bfseries Keywords: }{Hom-associative algebra, Hom-Lie algebra,
Hom-Lie superalgebra, Hom-Poisson algebra, Hom-Lie admissible
algebra, Hom-Lie admissible superalgebra, Hom-alternative algebra,
Hom-Jordan algebra.}

\medskip

\medskip

\section{Introduction}

We mean by Hom-algebra or Hom-superalgebra a triple consisting of a
vector space, respectively $\mathbb{Z}_2$-graded vector space, a
multiplication and a homomorphism. The main feature of these
algebras is that the identities defining the structures are twisted
by homomorphisms.  Such algebras appeared in the ninetieth in
examples of $q$-deformations of the Witt and the Virasoro algebras.

The study of nonassociative algebras was originally motivated by
certain problems in physics and other branches of mathematics. The
first motivation to study nonassociative Hom-algebras comes from
quasi-deformations of Lie algebras of vector fields, in particular
$q$-deformations of Witt and Virasoro algebras
\cite{AizawaSaito,ChaiElinPop,ChaiKuLukPopPresn,ChaiIsKuLuk,ChaiPopPres,
CurtrZachos1,DaskaloyannisGendefVir, Kassel1,LiuKeQin,Hu}. The
deformed algebras arising in connection with $\sigma$-derivation are
no longer Lie algebras. It was observed in the pioneering works
 that in these
examples a twisted Jacobi identity holds. Motivated by these
examples and their generalization,   Hartwig, Larsson and Silvestrov
introduced and studied  in \cite{HLS,LS1,LS2,LS3} the classes of
quasi-Lie, quasi-Hom-Lie and Hom-Lie algebras. In the class of
Hom-Lie algebras skew-symmetry is untwisted, whereas the Jacobi
identity is twisted by a homomorphism and contains three terms as in
Lie algebras, reducing to ordinary Lie algebras when the twisting
linear map is the identity map. They showed that Hom-Lie algebras are
closely related to discrete and deformed vector fields and
differential calculus and that some $q$-deformations of the Witt and
the Virasoro algebras have the structure of Hom-Lie algebra.

The  Hom-associative algebras plays the role of associative algebras
in the Hom-Lie setting. They were introduced by Makhlouf and
Silvestrov  in \cite{MS}, where it is shown that
 the commutator bracket of a Hom-associative algebra
gives rise to a Hom-Lie algebra and where a classification of Hom-Lie admissible
algebras is established.
 Given a Hom-Lie algebra,  a universal enveloping Hom-associative
 algebra was constructed by Yau in \cite{Yau:EnvLieAlg}.
 See also \cite{Canepl2009,
FregierGohr1,FregierGohr2,Gohr,HomDeform,Yau:homology} for other
works on Hom-associative algebras. In a similar way Yau proposed in
\cite{Yau:EnvLieAlg} a notion of Hom-dialgebra which gives rise to
Hom-Leibniz algebra. While Hom-associative superalgebras and Hom-Lie
superalgebras were studied in \cite{AmmarMakhlouf2009}.

The  Hom-alternative algebras and Hom-Jordan algebras which are
twisted version of the ordinary alternative algebras and Jordan
algebras were introduced by the author in
\cite{Mak:HomAlternativeHomJord}. Their properties are discussed and
 construction procedures using ordinary alternative algebras
or Jordan algebras are provided. Also, it is  shown that a plus
Hom-algebra  of a Hom-associative algebra leads to Hom-Jordan
algebra.

Beyond the binary algebras, a generalization of $n$-ary algebras to
Hom situation was studied in \cite{AMS2009}, $n$-ary Hom-algebras of
Lie type and associative type are discussed. Dualization of Hom-associative algebras leads to Hom-coassociative coalgebras. The Hom-coassociative coalgebras, Hom-bialgebras and Hom-Hopf algebras were introduced in \cite{HomHopf,HomAlgHomCoalg}. A twisted version of Yang-Baxter equation, quasi-triangular bialgebras and quantum groups were discussed in \cite{Yau:YangBaxter,Yau:YangBaxter2,Yau:ClassicYangBaxter,Yau:HomQuantumGrp1,Yau:HomQuantumGrp2}.
 A study from the monoidal category point of view is given in \cite{Canepl2009}.

The main purpose of this paper is to summarize  nonassociative
Hom-algebra structures extending the ordinary nonassociative
algebras to Hom-algebra setting. We set the definitions and
properties, and provide some examples. In  Section 2, we fix the
notations and some general setting of Hom-algebras. The third
Section concerns Hom-associative algebras, Hom-associative
superalgebras, Hom-flexible algebras and Hom-dialgebras. Section 4
deals with Hom-Lie algebras, Hom-Leibniz, Hom-Poisson algebras and
Hom-Lie superalgebras. It is shown in particular that there is  an
adjoint pair of functors between the category of Hom-associative
algebras (resp. Hom-dialgebras) and the category of Hom-Lie algebras
(resp. Hom-Leibniz algebras). On the other hand the supercommutator
of a Hom-associative superalgebra determines a Hom-Lie superalgebra.
The Hom-Poisson algebra structure which emerges naturally from
deformation theory of Hom-associative algebras (see
\cite{HomDeform}) is also given. We also present the definition of
quasi-Lie algebras which offer to treat within the same framework
the well known  generalizations of Lie algebras, that is Lie
superalgebras and  color Lie algebras as well as their corresponding
versions in Hom-superalgebra setting. Section 5 is dedicated to
Hom-Lie admissible algebras and Hom-Lie admissible superalgebras.
Their classifications lead to $G$-Hom-associative algebras and
$G$-Hom-associative superalgebras, which include in particular Hom-algebra versions of left-symmetric algebras (Vinberg algebras) or their opposite algebras, that is right symmetric algebras (pre-Lie algebras). In Section 6, we discuss
Hom-alternative algebras and study their properties. In particular,
we show that the twisted version of the associator  is an
alternating function of its arguments.  We show that an ordinary
alternative algebra and one of its algebra endomorphisms lead to a
Hom-alternative algebra where the twisting map is actually the
algebra endomorphism. This process was
  introduced in \cite{Yau:homology}
 for Lie and associative algebras and more generally  $G$-associative algebras
 (see \cite{MS} for this class of algebras) and generalized to coalgebras in \cite{HomAlgHomCoalg},
 \cite{HomHopf} and to $n$-ary algebras of Lie and associative types in
 \cite{AMS2009}. We derive examples of Hom-alternative algebras from
  4-dimensional alternative algebras which are
 not associative and from algebra of octonions. In the last Section
 we introduce a notion of Hom-Jordan algebras and show that
 it fits with the Hom-associative structure, that is the plus  Hom-associative algebra
 leads to Hom-Jordan algebra. {\ding{167}}

\section{Preliminaries}
Throughout this paper $\mathbb{K}$ is a field of characteristic 0,
$V$ is a $\K$-linear space or, when talking about superalgebras, $V$
is  a superspace. Let $V$ be a superspace over $\K$ that is a
$\mathbb{Z}_2$-graded $\K$-linear  space with a direct sum $V=V_0
\oplus V_1$. The element of $V_j$, $j=\{0,1\}$, are said to be
homogenous and of parity $j$. The parity of a homogeneous element
$x$ is denoted by $| x |$. One may consider that $\K$ is any
commutative ring and $V$ a $\K$-module.

We mean by a Hom-algebra, respectively Hom-superalgebra, a triple $(
V, \mu, \alpha) $ where  $\mu : V\times V \rightarrow V$ is a
$\K$-bilinear map and $\alpha:V \rightarrow V$ is a $\K$-linear map,
respectively an even $\K$-linear map. The type of the Hom-algebra or
the Hom-superalgebra is fixed by the identity satisfied by the
elements. Let $( V, \mu, \alpha) $ and $( V', \mu', \alpha') $ be
two Hom-algebras (resp. Hom-superalgebras) of the same type, a
morphism $f:( V, \mu, \alpha) \rightarrow ( V', \mu', \alpha')$ is a
linear map (resp. even linear map) $f: V \rightarrow
 V'$ such that $f\circ \mu=\mu' \circ (f\times f)$ and $f\circ
 \alpha=\alpha'\circ f$. In some statements the multiplicativity of $\alpha$ is required,
 that is $\alpha\circ \mu =\mu\circ (\alpha \times \alpha)$. We call such Hom-algebras,
 multiplicative Hom-algebras.

Let $V$ be a $n$-dimensional $\K$-linear space and
$\{e_1,\cdots,e_n\}$ be a basis of $V$. A  Hom-algebra structure on
$V$ with product $\mu$ is determined by $n^3$ structure constants
$C_{i j}^k$, where $\mu (e_i,e_j)=\sum_{k=0}^{n}{C_{i j}^k e_k}$ and
homomorphism $\alpha$ which is given by $n^2$ structure constants
$a_{i j}$, where $\alpha (e_i)=\sum_{j=0}^{n}{a_{i j} e_j}$. If we
require that the Hom-algebra is of a given type then this limits the
set of structure constants $(C_{i j}^k,a_{i j})$ to an algebraic
variety imbedded in $\K^{n^3+n^2}$. The polynomials defining the
algebraic variety being derived from the identities. A point in such
an algebraic variety  represents an $n$-dimensional Hom-algebra,
along with a particular choice of basis. A change of basis may give
rise to a  different point of the algebraic variety. The
group $GL(n,\K)$ acts on the algebraic varieties of Hom-structures
by the so-called "transport of structure" action defined as follows:

Let $\A=( V, \mu, \alpha)$ be a Hom-algebra. Given  $f\in GL(n,\K)$,
the action $f\cdot \A$ transports the structures for $x,y\in V$ by
\begin{eqnarray}
f\cdot \mu (x,y)=f^{-1}\mu (f(x),f(y))\\
f\cdot \alpha (x)=f^{-1}\alpha (f(x))
\end{eqnarray}
The \textit{orbit} of the Hom-algebra $\A$ is given by $\vartheta
\left( \A\right) =\left\{ \A^{\prime }=f\cdot \A,\quad f\in
GL_n\left( \K\right) \right\}.$ The orbits are in 1-1-correspondence
with the isomorphism classes of $n$-dimensional algebras. \newline
The stabilizer subgroup of $\A$ $\left( stab\left( \A\right)
=\left\{ f\in GL_n\left( \K\right) :\A=f\cdot \A\right\} \right) $
is exactly $Aut\left( \A\right) ,$ the automorphism group of $\A$.

A subalgebra of a Hom-algebra $( V, \mu, \alpha) $ is a triple $( W,
\mu, \alpha) $ where $W$ is a subspace of $V$ closed under $\mu$ and
$\alpha$. A subspace $I$ of $V$ is a two-sided ideal if $\mu (I,V)\subset I$ and
$\mu (V,I)\subset I$.

To any  Hom-algebra $\A=( V, \mu, \alpha) $, we associate  a minus
Hom-algebra $\A^-=( V, \mu^-, \alpha)$ where
$\mu^-(x,y)=\frac{1}{2}(\mu(x,y)-\mu(y,x))$ and a plus Hom-algebra
$\A^+=( V, \mu^+, \alpha)$ where
$\mu^+(x,y)=\frac{1}{2}(\mu(x,y)+\mu(y,x))$.

 We call \emph{Hom-associator} associated
to a Hom-algebra $(V,\mu,\alpha)$  the trilinear map
$\mathfrak{as_\alpha}$ defined for any $x,y,z\in V$ by
\begin{equation}\label{HomAssociator}
\mathfrak{as}_\alpha(x,y,z)=\mu (\alpha(x),\mu (y, z))-\mu(\mu(x,
y), \alpha(z)).{\ding{167}}
\end{equation}

\section{Hom-associative algebras, Hom-Flexible algebras,
 Hom-dialgebras}
\subsection{Hom-associative algebras and superalgebras}
The notion of Hom-associative algebra was introduced in \cite{MS}, while the Hom-associative algebras were discussed slightly in \cite{AmmarMakhlouf2009}.
\begin{definition}[\cite{MS}]
A Hom-associative algebra  is a triple $( V, \mu,
\alpha) $ consisting of a vector space $V$, a bilinear map $\mu : V\times V \rightarrow V$
and a linear map $\alpha:V \rightarrow V$  satisfying for all $x,y,z\in V$
\begin{equation}\label{Hom-ass}
\mu(\alpha(x), \mu (y, z))= \mu (\mu (x, y), \alpha (z)).
\end{equation}
\end{definition}
In terms of Hom-associator \eqref{HomAssociator},  the identity
\eqref{Hom-ass} writes $\mathfrak{as}_\alpha(x,y,z)=0$.
\begin{example}\label{example1ass}
Let $\{e_1,e_2,e_3\}$  be a basis of a $3$-dimensional linear space
$V$ over $\K$. The following multiplication $\mu$ and linear map
$\alpha$ on $V$ define Hom-associative algebras over $\K^3${\rm :}
$$
\begin{array}{ll}
\begin{array}{lll}
 \mu ( e_1,e_1)&=& a\ e_1, \ \\
\mu ( e_1,e_2)&=&\mu ( e_2,e_1)=a\ e_2,\\
\mu ( e_1,e_3)&=&\mu ( e_3,x_1)=b\ e_3,\\
 \end{array}
 & \quad
 \begin{array}{lll}
\mu ( e_2,e_2)&=& a\ e_2, \ \\
\mu ( e_2, e_3)&=& b\ e_3, \ \\
\mu ( e_3,e_2)&=& \mu ( e_3,e_3)=0,
  \end{array}
\end{array}
$$

$$  \alpha (e_1)= a\ e_1, \quad
 \alpha (e_2) =a\ e_2 , \quad
   \alpha (e_3)=b\ e_3,
$$
where $a,b$ are parameters in $\K$. The algebras are not associative
when $a\neq b$ and $b\neq 0$, since
$$\mu (\mu (e_1,e_1),e_3))- \mu ( e_1,\mu
(e_1,e_3))=(a-b)b e_3.$$

\end{example}

\begin{example}[Polynomial Hom-associative algebra \cite{Yau:homology}]
Consider the polynomial algebra $\A=\K [x_1,\cdots x_n]$ in $n$
variables. Let $\alpha$ be an algebra endomorphism of $\A$ which is
uniquely determined by the $n$ polynomials $\alpha (x_i)=\sum
{\lambda_{i;r_1,\cdots,r_n}x^{r_1}_1,\cdots x^{r_n}_n}$ for $1\leq i
\leq n$. Define $\mu$ by
\begin{equation}
\mu (f,g)=f(\alpha (x_1),\cdots \alpha (x_n))g(\alpha (x_1),\cdots
\alpha (x_n))
\end{equation}
for $f,g$ in $\A$. Then,  $(\A ,\mu,\alpha)$ is a Hom-associative algebra.
\end{example}
\begin{example}[\cite{Yau:comodule}]
Let $\A=(V,\mu,\alpha )$ be a Hom-associative algebra. Then
$(\mathcal{M}_n(\A),\mu',\alpha' )$, where $\mathcal{M}_n(\A)$ is
the vector space of $n\times n$ matrix with entries in $V$, is also
a Hom-associative algebra in which the multiplication $\mu'$ is
given by matrix multiplication and $\mu$ and $\alpha'$ is given by
$\alpha$ in each entry.

\end{example}

We define now  Hom-associative superalgebras.
\begin{definition}[\cite{AmmarMakhlouf2009}]
A Hom-associative superalgebra is a triple $(V, \mu, \alpha)$
consisting of a superspace $V$, an even bilinear map $\mu: V\times V
\rightarrow V$ and an even  homomorphism $\alpha: V \rightarrow V$
satisfying for all $x,y,z\in V$
\begin{equation}
\mu(\alpha(x),\mu (y,z))=\mu (\mu (x,y),\alpha (z))
\end{equation}
\end{definition}

\begin{remark}
Some properties of  Hom-associative algebras and unital
Hom-asso-ciative algebras were discussed in
\cite{AmmarMakhlouf2009,AMS2009,
FregierGohr1,FregierGohr2,Gohr,HomDeform,Yau:homology}. A study from
the point of view of monoidal category was provided in
\cite{Canepl2009} and where a new kind of unital Hom-associative
algebras is introduced.
\end{remark}

\subsection{Hom-Flexible algebras}
  The flexible algebras was initiated by Albert
  \cite{Albert48} and investigated by number of authors Myung,
  Okubo, Laufer, Tomber and Santilli,  see for example \cite{MyungMonograph}.
  The notion of Hom-flexible algebra was introduced in \cite{MS}. We
  summarize in the following the definition and some
  characterizations.
 \begin{definition}[\cite{MS}]
 A Hom-algebra $\A =(V, \mu, \alpha)$ is called
 flexible if and only if   for any $x,y$ in $V$
\begin{equation}\label{flexible}
      \mu (\alpha (x),\mu (y,x)))=\mu (\mu (x,y), \alpha (x))
\end{equation}

 \end{definition}
 The condition (\ref{flexible}) may be written using the Hom-associator by
$
      \mathfrak{as_\alpha}(x,y,x)=0.
$

We recover the classical dialgebra when $\alpha$ is the identity
map.
\begin{lemma}\label{flexi23}
Let $\A =(V, \mu, \alpha)$ be a Hom-algebra. The
following assertions are equivalent
\begin{enumerate}
\item $\A$ is flexible.
\item For any $x,y$ in $V$, $\mathfrak{as_\alpha}(x,y,x)=0$.
\item For any $x,y,z$ in $V$,
      $\mathfrak{as_\alpha}(x,y,z)=-\mathfrak{as_\alpha}(z,y,x)$
\end{enumerate}
\end{lemma}
\begin{proof}
The first equivalence follows from the definition. To prove the last
equivalence, one writes
$ \mathfrak{as_\alpha}(x-z,y,x-z)=0
$
which is equivalent, using the linearity, to
$
\mathfrak{as_\alpha}(x,y,z)+\mathfrak{as_\alpha}(z,y,x)=0.
$
\end{proof}
\begin{corollary}
Any Hom-associative algebra is flexible.
\end{corollary}
Let $\A =(V,\mu , \alpha)$ be a Hom-algebra. Let $\A ^-=(V,\mu^- ,
\alpha)$
  (resp. $\A ^+=(V,\mu^+ , \alpha)$) be the
 plus Hom-algebra (resp. minus Hom-algebra) with multiplication defined for
 $x,y\in V$ by $\mu^-(x,y)=\frac{1}{2}(\mu(x,y)-\mu(y,x))$ (resp.
$\mu^+(x,y)=\frac{1}{2}(\mu(x,y)+\mu(y,x))$). We have the following  characterization of
Hom-flexible algebras.

\begin{proposition}[\cite{MS}] A Hom-algebra $\A =(V,\mu , \alpha)$ is Hom-flexible if and
only if \begin{equation} \label{equa2}
  \mu^-(\alpha(x),\mu^+(y, z))=\mu^+(\mu^-(x,y),\alpha(z))+
  \mu^+(\alpha(y),\mu^-(x,z))
\end{equation}

\end{proposition}
\begin{proof}
Let $\A$ be a Hom-flexible algebra. Then by lemma \ref{flexi23} it
is equivalent to
      $\mathfrak{as_\alpha}(x,y,z)+\mathfrak{as_\alpha}(z,y,x)=0$, for any $x,y,z$ in $V$,
       where  $\mathfrak{as_\alpha}$ is the Hom-associator associated to $\A$.
 This implies
\begin{eqnarray}\label{equa} &&\mathfrak{as_\alpha}(x,y,z)+
\mathfrak{as_\alpha}(z,y,x)+\mathfrak{as_\alpha}(x,z,y)
\\&& +\mathfrak{as_\alpha}(y,z,x)-\mathfrak{as_\alpha}(y,x,z)-\mathfrak{as_\alpha}(z,x,y)=0\nonumber
 \end{eqnarray}
By expansion, the previous relation is equivalent to (\ref{equa2}).

Conversely, assume that we have the condition (\ref{equa2}), by
setting $x=z$ in the equation (\ref{equa}), one gets
$\mathfrak{as_\alpha}(x,y,x)=0$, Therefore $\A$ is a Hom-flexible algebra.
\end{proof}

\subsection{Hom-dialgebras}
 The Hom-dialgebra structure introduced by Yau extends to Hom-algebra setting
the classical dialgebra structure introduced by Loday.
\begin{definition}[\cite{Yau:EnvLieAlg}]
A Hom-dialgebra is a tuple $(V,\dashv,\vdash,\alpha)$, where
$\dashv,\vdash : V\times V\rightarrow V$ are bilinear maps and
$\alpha : V\rightarrow V$ is a linear map  such that the following
five identities are satisfied for $x,y,z\in V$
\begin{eqnarray}
\alpha (x)\dashv (y\dashv z)&=&(x\dashv y)\dashv \alpha (z)=\alpha
(x)\dashv (y\vdash z)\\
\alpha (x)\vdash (y\vdash z)&=&(x\vdash y)\vdash \alpha (z)=\alpha
(x)\vdash (y\dashv z)\\
\alpha (x)\vdash (y\dashv z)&=&(x\vdash y)\dashv \alpha (z)
\end{eqnarray}
\end{definition}
We recover the classical dialgebra when $\alpha$ is the identity
map. A morphism of Hom-dialgebras is a linear map that is compatible
with $\alpha$ and the two multiplications $\vdash$ and $\dashv$.

Note that $(V,\dashv,\alpha)$ and $(V,\vdash,\alpha)$ are
Hom-associative algebras. In the classical case, Loday showed that
the commutator defined for $x,y\in V$ by $[x,y]=x\dashv y-y\vdash x$
defines a Leibniz algebra on $V$. This result is extended to
Hom-algebra setting in the next Section.{\ding{167}}
\section{Hom-Leibniz algebras and Hom-Lie algebras}
\subsection{Hom-Leibniz algebras}

A class of quasi Leibniz algebras was introduced in \cite{LS2} in
connection to general quasi-Lie algebras following the standard
Loday's conventions for Leibniz algebras (i.e. right Loday
algebras).

\begin{definition}[\cite{LS2}]
A Hom-Leibniz algebra is a triple $(V, [\cdot, \cdot], \alpha)$
consisting of a linear space $V$, bilinear map $[\cdot, \cdot]:
V\times V \rightarrow V$ and a linear map $\alpha: V \rightarrow
V$  satisfying
\begin{equation} \label{Leibnizalgident}
 [[x,y],\alpha(z)]=[[x,z],\alpha (y)]+[\alpha(x),[y,z]]
\end{equation}
\end{definition}

In terms of the (right) adjoint homomorphisms $Ad_y: V\rightarrow V$
defined by $Ad_{y}(x)=[x,y]$, the identity \eqref{Leibnizalgident}
can be written as
\begin{equation} \label{LeibnizalgidentAdDeriv}
Ad_{\alpha(z)}([x,y]) = [Ad_{z}(x),\alpha(y)] +[\alpha(x),Ad_{z}(y)]
\end{equation}
or in pure operator form
\begin{equation} \label{LeibnizalgidentAdOper}
Ad_{\alpha(z)} \circ Ad_{y} = Ad_{\alpha(y)} \circ Ad_{z} +
Ad_{Ad_{z}(y)} \circ \alpha
\end{equation}

\subsection{Hom-Lie algebras}
The Hom-Lie algebras were initially introduced by Hartwig, Larson
and Silvestrov in \cite{HLS} motivated initially by examples of
deformed Lie algebras coming from twisted discretizations of vector
fields.

\begin{definition}[\cite{HLS}]
A Hom-Lie algebra is a triple $(V, [\cdot, \cdot], \alpha)$
consisting of
 a linear space $V$, bilinear map $[\cdot, \cdot]: V\times V \rightarrow V$ and
 a linear map $\alpha: V \rightarrow V$
 satisfying
\begin{eqnarray}\label{skewsymmetry} [x,y]=-[y,x] \quad
{\text{(skewsymmetry)}} \\ \label{HomJacobi}
\circlearrowleft_{x,y,z}{[\alpha(x),[y,z]]}=0 \quad
{\text{(Hom-Jacobi identity)}}
\end{eqnarray}
for all $x, y, z$ from $V$, where $\circlearrowleft_{x,y,z}$ denotes
summation over the cyclic permutation on $x,y,z$.
\end{definition}
Using the skew-symmetry, one may write the Hom-Jacobi identity in
the form (\ref{LeibnizalgidentAdDeriv}). Hence, if a Hom-Leibniz
algebra is skewsymmetric then it is a Hom-Lie algebra.

\begin{example}
Let $\{e_1,e_2,e_3\}$  be a basis of a $3$-dimensional linear space
$V$ over $\K$. The following bracket and   linear map $\alpha$ on
$V$ define a Hom-Lie algebra over $\K^3${\rm :}
$$
\begin{array}{cc}
\begin{array}{ccc}
 [ e_1, e_2 ] &= &a e_1 +b e_3 \\ {}
 [e_1, e_3 ]&=& c e_2  \\ {}
 [ e_2,e_3 ] & = & d e_1+2 a e_3,
 \end{array}
 & \quad

  \begin{array}{ccc}
  \alpha (e_1)&=&e_1 \\
 \alpha (e_2)&=&2 e_2 \\
   \alpha (e_3)&=&2 e_3
  \end{array}

\end{array}
$$
with $[ e_2, e_1 ]$, $[e_3, e_1 ]$ and  $[ e_3,e_2 ]$ defined via
skewsymmetry. It is not a Lie algebra if and only if $a\neq0$ and
$c\neq0$, since
$$[e_1,[e_2,e_3]]+[e_3,[e_1,e_2]]
+[e_2,[e_3,e_1]]= a c e_2.$$
\end{example}

\begin{example}[\textbf{Jackson $\sll$}]

In this example, we will consider the Hom-Lie algebra Jackson $\sll$
which is a Hom-Lie deformation of the classical Lie algebra  $\sll$
defined by $[ e_1,e_2] = 2e_2,\  [ e_1,e_3] = -2e_3 , \  [ e_2,e_3]
=e_1$. This family of Hom-Lie algebras was constructed in \cite{LS3}
using  a quasi-deformation scheme based on discretizing by means of
Jackson $q$-derivations a representation of $\sll$ by
one-dimensional vector fields (first order ordinary differential
operators) and using the twisted commutator bracket defined in
\cite{HLS}. The Hom-Lie algebra Jackson $\sll$ is a $3$-dimensional
vector space over $\K$ with the vector space bases $\{e_1,e_2,e_3\}$ and with the bilinear bracket multiplication defined
on the basis by
$$[ e_1,e_2]_t = 2e_2,\ \
[ e_1,e_3]_t = -2e_3-2te_3 , \ \ [ e_2,e_3]_t
=e_1+\frac{t}{2}e_1,$$ and by skew-symmetry for
$[ e_2,e_1]_t, [ e_3,e_1]_t$ and $[e_3,e_2]_t$.
The linear map $\alpha_t$ is defined by
$$\alpha_t(e_1)=
e_1, \quad \alpha_t(e_2)=\frac{2+t}{2(1+t)}e_2=e_2+
\sum_{k=0}^{\infty}{\frac{(-1)^k}{2}t^k\ } e_2, \quad
\alpha_t(e_3)=e_3+\frac{t}{2} \, e_3.$$
%}

Thus Jackson $\sll$ algebra is a Hom-Lie algebra
deformation of $\sll$. See \cite{HomDeform} for other examples of Hom-Lie algebra
deformation of $\sll$.
\end{example}

\subsection{Hom-Lie and Hom-Leibniz Functors}

 We provide in the following a different way for constructing Hom-Lie
algebras by extending the fundamental construction of Lie algebras
from associative algebras via commutator bracket.
\begin{theorem}[\cite{MS}]
Let $(V,\mu,\alpha)$ be a Hom-associative algebra. The Hom-algebra
 $(V,[\cdot,\cdot],\alpha)$, where the bracket is defined for all $x,y \in
V$ by
$$
[ x,y ]=\mu (x,y)-\mu (y,x )
$$
is a Hom-Lie algebra.
\end{theorem}
\begin{proof}
The bracket is obviously skewsymmetric and with a direct computation
we have
$$
\begin{array}{c}
  [\alpha (x),[y,z]]-[[x,y],\alpha (z)]-[\alpha(y),[y,z]]= \\
  \mu(\alpha (x),\mu(y,z))-\mu(\alpha(x),\mu(z,y))-\mu(\mu(y,z),\alpha(x))+
\mu(\mu(z,y),\alpha(x))\\
-\mu(\mu(x,y),\alpha(z))+\mu(\mu(y,x),\alpha(z))+\mu(\alpha
(z),\mu(x,y))-\mu(\alpha (z),\mu(y,x))\\ -\mu(\alpha
(y),\mu(x,z))+\mu(\alpha
(y),\mu(z,x))+\mu(\mu(x,z),\alpha(y))-\mu(\mu(z,x),\alpha(y))=0
\end{array}
$$
\end{proof}

A similar construction is obtained for Hom-dialgebra.
\begin{theorem}[\cite{Yau:EnvLieAlg}]
Let $(V,\dashv,\vdash,\alpha)$ be a Hom-dialgebra $(V,\mu,\alpha)$. The Hom-algebra $(V,[\cdot,\cdot],\alpha)$, where the bracket is defined for all $x,y \in
V$ by
$$
[x,y]=x\dashv y-y\vdash x$$
is a Hom-Leibniz algebra.
\end{theorem}

Therefore, we have a functor $HLie$ (resp. $HLeib$) from the category of Hom-associative algebras $\textbf{HomAs}$ (resp. category of Hom-dialgebras $\textbf{HomDi}$) to the category of Hom-Lie algebras $\textbf{HomLie}$ (resp. category of Hom-Leibniz algebras $\textbf{HomLeib}$).
Conversely, an enveloping Hom-associative algebra $U_{HLie}(\mathfrak{g})$
(resp. enveloping Hom-dialgebra $U_{HLeib}(L)$) of  a Hom-Lie algebra $\mathfrak{g}$ (resp. Hom-Leibniz algebra $L$) are constructed in \cite{Yau:EnvLieAlg}. Hence, $U_{HLie}$ is the left adjoint functor of $HLie$ and $U_{HLeib}$ is the left adjoint functor of $HLeib$.
\subsection{Hom-Poisson algebras}
We introduce in the following the notion of Hom-Poisson structure
which emerges naturally in   deformation theory of Hom-associative
algebras, see \cite{HomDeform}.

\begin{definition}[\cite{HomDeform}]
A \emph{Hom-Poisson algebra} is a quadruple $(V,\mu,
\{\cdot,\cdot\}, \alpha)$ consisting of
 a vector space $V$, bilinear maps $\mu: V\times V \rightarrow V$ and
  $\{\cdot, \cdot\}: V\times V \rightarrow V$, and
 a linear map $\alpha: V \rightarrow V$
 satisfying
 \begin{enumerate}
\item $(V,\mu, \alpha)$ is a commutative Hom-associative algebra,
\item $(V, \{\cdot,\cdot\}, \alpha)$ is a Hom-Lie algebra,
\item
for all $x, y, z$ in $V$,
\begin{equation}\label{CompatibiltyPoisson}
\{\alpha (x) , \mu (y,z)\}=\mu (\alpha (y), \{x,z\})+ \mu (\alpha
(z), \{x,y\}).
\end{equation}
\end{enumerate}
\end{definition}
 Condition \eqref{CompatibiltyPoisson}
expresses the compatibility between the multiplication and the
Poisson bracket. It can be
 reformulated equivalently, for all $x, y, z$ in $V$, as
\begin{equation}\label{CompatibiltyPoissonLeibform}
\{\mu(x,y),\alpha (z) \}=\mu (\{x,z\},\alpha (y))+\mu (\alpha (x),
\{y,z\})
\end{equation}
Note that in this form it means that $ad_z
(\cdot) = \{\cdot,z\}$ is a sort of generalization of a derivation
of an associative  algebra, and also it resembles the identity
\eqref{Leibnizalgident} in the definition of Leibniz algebra.
We recover the classical Leibniz identity when $\alpha$ is the identity map.

\begin{example}\label{example1HomPoisson}
Let $\{e_1,e_2,e_3\}$  be a basis of a $3$-dimensional vector space
$V$ over $\K$. The following multiplication $\mu$, skew-symmetric
bracket and linear map $\alpha$ on $V$ define a Hom-Poisson algebra
over $\K^3${\rm :}
$$
\begin{array}{ll}
\begin{array}{lll}
 \mu ( e_1,e_1)&=&  e_1, \ \\
\mu ( e_1,e_2)&=& \mu ( e_2,e_1)=e_3,\\
 \end{array}
 & \quad
 \begin{array}{lll}
\{ e_1,e_2 \}&=& a e_2+ b e_3, \ \\
\{ e_1, e_3 \}&=& c e_2+ d e_3, \ \\
  \end{array}
\end{array}
$$

$$  \alpha (e_1)= \lambda_1 e_2+\lambda_2 e_3 , \quad
 \alpha (e_2) =\lambda_3 e_2+\lambda_4 e_3  , \quad
   \alpha (e_3)=\lambda_5 e_2+\lambda_6 e_3
$$
where $a,b,c,d,\lambda_1,
\lambda_2,\lambda_3,\lambda_4,\lambda_5,\lambda_6 $ are parameters
in $\K$.
\end{example}

\subsection{Hom-Lie Superalgebras} Hom-Lie Superalgebras is a
subclass of quasi-Lie algebras introduced in \cite{LS2}. They were
studied in \cite{AmmarMakhlouf2009}, where construction procedures
are provided.
\begin{definition}
A Hom-Lie superalgebra is a triple $(V, [\cdot, \cdot], \alpha)$
consisting of
 a superspace $V$, an even bilinear map $[\cdot, \cdot]: V\times V \rightarrow V$ and
 an even superspace homomorphism $\alpha: V \rightarrow V$
 satisfying
\begin{eqnarray} \label{skewsymHomLie}
 &[x,y]=-(-1)^{\mid x\mid\mid y\mid}[y,x]\\
\label{JacobyHomsuperLie}
 &(-1)^{\mid x\mid\mid z\mid}[\alpha(x),[y,z]]+(-1)^{\mid z\mid\mid y\mid}
 [\alpha(z),[x,y]]+(-1)^{\mid y \mid\mid x\mid}[\alpha(y),[z,x]]=0
\end{eqnarray}
for all homogeneous element $x, y, z$ in $V$.
\end{definition}

The identity \eqref{JacobyHomsuperLie} is called Hom-superJacobi identity, while the identity
\eqref{JacobyHomsuperLie} expresses the usual supersymmetry of the bracket.
\begin{remark}
We recover the classical Lie superalgebra when $\alpha
=\id$. The Hom-Lie algebras  algebras are obtained when the part of
parity one is trivial.
\end{remark}
\begin{example}[2-dimensional abelian Hom-Lie  superalgebra]

Every  bilinear map $\mu$ on a $2$- dimensional linear superspace
$V=V_0 \oplus V_1$, where $V_0$ is generated by $x$ and $V_1$ is
generated by $y$ and  such that $[x,y]=0$ defines a Hom-Lie
superalgebra for any homomorphism $\alpha$ of superalgebra. Indeed,
the graded Hom-Jacobi identity is satisfied for any triple
$(x,x,y)$.
\end{example}

\begin{example}[Affine Hom-Lie  superalgebra ]
Let $V=V_0 \oplus V_1$ be a $3$-dimensional superspace where $V_0$
is generated by $\{e_1 , e_2\}$ and $V_1$ is generated by $e_3$. The
triple
 $(V,[\cdot,\cdot],\alpha)$ is a Hom-Lie superalgebra defined by $[e_1
 ,e_2]=e_1,[e_1,e_3]=[e_2,e_3]=[e_3,e_3]=0$ and $\alpha$ is any
 homomorphism.
\end{example}

\begin{example}[A $q$-deformed  Witt superalgebra, \cite{AmmarMakhlouf2009}]
We provide an example of infinite dimensional Hom-Lie
superalgebra which is given by a realization of the $q$-deformed Witt superalgebra
constructed in
\cite{AmmarMakhlouf2009}. It corresponds to  a superspace $\mathcal{V}$   generated by the
  elements $\{X_n\}_{n\in \mathbb{N}}$
  of parity $0$
   and the elements $\{G_n\}_{n\in \mathbb{N}}$ of parity $1$.

 Let $q\in \mathbb{C}\backslash \{0,1\}$ and $n\in\mathbb{N}$, we set
 $\{n\}=\frac{1-q^n}{1-q}$, a $q$-number. The $q$-numbers have the following
 properties $\{n+1\}=1+q \{n\}=\{n\}+q^n$ and $\{n+m\}=\{n\}+q^n\{m\}.$

  Let $[-,-]_{\sigma}$ be a bracket on the
  superspace $\mathcal{V}$  defined by
  \begin{align*}
[ X_n,X_m]_\sigma &=(\{m\}-\{n\})X_{n+m} \\
 [X_n,G_m]_\sigma &= (q^n \{m+1\}-q^{m+1}\{n\}) G_{n+m}
\end{align*}
The others brackets are obtained by supersymmetry or are $0$.

Let $\alpha$ be an even linear map on $\mathcal{V}$ defined
on the generators by
\begin{align*}
\alpha ( X_n ) &=(1+ q^n)X_{n} \\
 \alpha ( G_n ) &=(1+ q^{n+1})G_{n}
\end{align*}
Then the triple  $(\mathcal{V},[-,-]_{\sigma},\alpha)$
   is a Hom-Lie
  superalgebra.

\end{example}

In the following, we show that the supercommutator bracket defined
using the multiplication in a Hom-associative superalgebra leads
naturally to Hom-Lie superalgebra.
\begin{theorem}[\cite{AmmarMakhlouf2009}]\label{Supercommutator}
Let $(V,\mu,\alpha )$ be a Hom-associative superalgebra. The Hom-superalgebra $(V,[\cdot,\cdot],\alpha )$, where the bracket (super bracket) is defined for all homogeneous elements $x,y\in V$ by
   $$ [x,y] = \mu (x,y) - ( - 1)^{ | x | | y |} \mu (y,x)$$
and extended by linearity to all elements,  is a Hom-Lie superalgebra.
\end{theorem}
\begin{proof}
The bracket is obviously supersymmetric and with a direct
computation we have
$$
\begin{array}{c}
  (-1)^{|x||z|}[\alpha (x),[y,z]]+
  (-1)^{|z||y|}[\alpha (z),[x,y]]+
  (-1)^{|y||x|}[\alpha(y),[z,x]]= \\
  (-1)^{|x||z|}\mu(\alpha (x),\mu(y,z))-(-1)^{|x||z|+|y||z|}\mu(\alpha(x),\mu(z,y))
  \\-(-1)^{|x||y|}\mu(\mu(y,z),\alpha(x))+(-1)^{|x||y|+|y||z|}
\mu(\mu(z,y),\alpha(x))\\
+(-1)^{|y||x|}\mu(\alpha
(y),\mu(z,x))-(-1)^{|x||y|+|z||x|}\mu(\alpha(y),\mu(x,z))
 \\ -(-1)^{|y||z|}\mu(\mu(z,x),\alpha(y))+(-1)^{|y||z|+|z||x|}
\mu(\mu(x,z),\alpha(y))\\
+(-1)^{|z||y|}\mu(\alpha(z),\mu(x,y))-(-1)^{|y||z|+|x||y|}\mu(\alpha(z),\mu(y,x))
  \\-(-1)^{|z||x|}\mu(\mu(x,y),\alpha(z))+(-1)^{|z||x|+|x||y|}
\mu(\mu(y,x),\alpha(z)) =0
\end{array}
$$
\end{proof}
The following theorem gives a way to construct Hom-Lie
superalgebras, starting from a regular  Lie superalgebra and an even
superalgebra endomorphism.

\begin{theorem}[\cite{AmmarMakhlouf2009}]
\label{thm:SALmorphism}
Let $(V,[\cdot,\cdot])$ be a Lie superalgebra  and  $\alpha :
V\rightarrow V$ be an even  Lie superalgebra endomorphism. Then
$(V,[\cdot,\cdot]_\alpha,\alpha)$, where
$[x,y]_\alpha=\alpha([x,y])$,  is a Hom-Lie superalgebra.

Moreover, suppose that  $(V',[\cdot,\cdot]')$ is another Lie
superalgebra and  $\alpha ' : V'\rightarrow V'$ is a Lie
superalgebra endomorphism. If $f:V\rightarrow V'$ is a Lie
superalgebra morphism that satisfies $f\circ\alpha=\alpha'\circ f$
then
$$f:(V,[\cdot,\cdot]_\alpha,\alpha)\longrightarrow (V',[\cdot,\cdot]',\alpha ')
$$
is a morphism of Hom-Lie superalgebras.
\end{theorem}
%\begin{proof}
%We show that $(V,[\cdot,\cdot]_\alpha,\alpha)$ satisfies the
%Hom-superJacobi
%  identity \ref{JacobyHomsuperLie}. Indeed
%\begin{align*}
%\circlearrowleft_{x,y,z}{(-1)^{\mid x \mid\mid
%z\mid}[\alpha(x),[y,z]_\alpha]_
%\alpha}&=\circlearrowleft_{x,y,z}{(-1)^{\mid x\mid\mid z\mid}\alpha ([\alpha(x),\alpha([y,z])])}\\
%&=\alpha^2(\circlearrowleft_{x,y,z}{(-1)^{\mid x\mid\mid z\mid} [x,[y,z]])}\\
%&=\alpha^2(\circlearrowleft_{x,y,z}{(-1)^{\mid x\mid\mid z\mid} [x,[y,z]])}\\
%&=0
%\end{align*}
%
%The second assertion follows from
%$$ f\circ [\cdot,\cdot]_\alpha = f\circ \alpha \circ [\cdot,\cdot]
%= \alpha ' \circ f \circ [\cdot,\cdot] = \alpha ' \circ
%[\cdot,\cdot]' \circ f = [\cdot,\cdot]'_{\alpha '} \circ f. $$
%\end{proof}

\begin{example}[\cite{AmmarMakhlouf2009}]We construct an example of Hom-Lie superalgebra, which is not
a Lie superalgebra starting from the orthosymplectic Lie
superalgebra. We consider in the sequel the matrix realization of
this Lie superalgebra.

Let $osp(1,2)=V_0 \oplus V_1$  \ be  the Lie superalgebra where
$V_0$ is generated by:
$$ H=\left(
  \begin{array}{ccc}
  1 & 0& 0 \\
  0 &0 & 0 \\
    0 & 0 & -1\\
  \end{array}\right), \ \ X=\left(
  \begin{array}{ccc}
  0 & 0 & 1 \\
  0 & 0 & 0 \\
  0 & 0 & 0\\
  \end{array}
\right),\ \ Y=\left(
  \begin{array}{ccc}
  0 & 0& 0 \\
  0 & 0 & 0 \\
  1 & 0 & 0\\
  \end{array}
\right), $$ and $V_1$ is generated by:

$$ F=\left(
  \begin{array}{ccc}
  0 & 0 & 0 \\
  1 & 0 & 0 \\
  0 & 1 & 0\\
  \end{array}
\right) , \ \ G=\left(
  \begin{array}{ccc}
  0 & 1& 0 \\
  0 & 0 & -1 \\
  0 & 0 & 0\\
  \end{array}
\right) .$$

The defining relations (we give only the ones with non zero values
in the right hand side) are
$$[H,X]=2X, \ [H,Y]=-2Y,\ [X,Y]=H,$$
$$[Y,G]=F,\ [X,F]=G,\ [H,F]=-F, \ [H,G]=G,$$
$$\ [G,F]=H, \ [G,G]=-2X,\ [F,F]=2Y.
$$
 Let $\lambda \in \mathbb{R}^*$, we consider the linear map
$\alpha_{\lambda}: osp(1,2)\rightarrow osp(1,2)$ defined by:
$$
\alpha_{\lambda}(X)=\lambda^2 X, \ \
\alpha_{\lambda}(Y)=\frac{1}{\lambda^2}Y, \ \
\alpha_{\lambda}(H)=H,\ \alpha_{\lambda}(F)=\frac{1}{\lambda}F, \ \
\alpha_{\lambda}(G)=\lambda G .$$

We provide a family of Hom-Lie superalgebras
$osp(1,2)_\lambda=(osp(1,2),[\cdot,\cdot]_{\alpha_\lambda} ,
\alpha_\lambda)$ where the Hom-Lie superalgebra bracket
$[\cdot,\cdot]_{\alpha_\lambda}$ on the basis elements is given, for
$\lambda\neq 0$, by:
$$[H,X]_{\alpha_\lambda}=2\lambda^2 X,\ \
[H,Y]_{\alpha_\lambda}=-\frac{2}{\lambda^2}Y, \ \
[X,Y]_{\alpha_\lambda}=H,$$ $$
[Y,G]_{\alpha_\lambda}=\frac{1}{\lambda}F,\ \
[X,F]_{\alpha_\lambda}=\lambda G,\ \
[H,F]_{\alpha_\lambda}=-\frac{1}{\lambda}F,\ \
[H,G]_{\alpha_\lambda}=\lambda G,$$ $$ [G,F]_{\alpha_\lambda}=H,\ \
[G,G]_{\alpha_\lambda}=-2\lambda^2X,\ \ [F,F]_{\alpha_\lambda}=
\frac{2}{\lambda^2}Y. $$

These Hom-Lie superalgebras are not Lie superalgebras for
$\lambda\neq 1$.

Indeed, the left hand side of the superJacobi identity
(\ref{JacobyHomsuperLie}), for $\alpha =\id$,  leads to
$$[X,[Y,H]-[H,[X,Y]]+[Y,[H,X]]=\frac{2(1-\lambda ^4)}{\lambda^{2}}Y,
$$
and also
$$[H,[F,F]-[F,[H,F]]+[F,[F,H]]=\frac{4(\lambda -1)}{\lambda^{4}}Y.
$$
Then, they do not vanish for $\lambda\neq 1.$

\end{example}

\subsection{Quasi-Lie algebras}
The class of quasi-Lie algebras where introduced by Larsson and
Silvestrov in order to treat within the same framework such a well
known generalizations of Lie algebras as color and Lie
superalgebras, as well as Hom-Lie algebras (see \cite{LS1,LS2}).

Let $\mathfrak{L}_{\K}(V)$ be the set of linear maps of the linear
space $L$ over the field $\K$.
\begin{definition} [ \cite{LS2}] \label{def:quasiLiealg}
A \emph{quasi-Lie algebra} is a tuple

 $(V,[\cdot,\cdot],\alpha,\beta,\omega,\theta)$ consisting of
\begin{itemize}
    \item $V$ is a linear space over $\mathbb{K}$,
    \item $[ \cdot,\cdot]:V\times V\to V$ is a bilinear
      map called a product or bracket in $V$;
    \item $\alpha,\beta:V\to V$, are linear maps,
    \item $\omega:D_\omega\to \mathfrak{L}_{\mathbb{K}}(V)$ and
    $\theta:D_\theta\to \mathfrak{L}_{\mathbb{K}}(V)$
      are maps with domains of definition
      $D_\omega, D_\theta\subseteq V\times
    V$,
\end{itemize}
such that the following conditions hold:
\begin{itemize}
      \item ($\omega$-symmetry) The product satisfies a generalized skew-symmetry condition
        $$[ x,y]=\omega(x,y)[ y,x],
        \quad\text{ for all } (x,y)\in D_\omega ;$$
\item (quasi-Jacobi identity) The bracket satisfies a generalized Jacobi identity
    $$\circlearrowleft_{x,y,z}\big\{\,\theta(z,x)\big([\alpha(x),[ y,z]]+
    \beta [ x,[ y,z] ]\big)\big\}=0,$$
     for all $(z,x),(x,y),(y,z)\in D_\theta$.
\end{itemize}
\end{definition}
Note that $(\omega(x,y)\omega(y,x)-id)[ x,y]=0,$ if $(x,y), (y,x)
\in D_\omega$, which follows from the computation $[
x,y]=\omega(x,y)[ y,x] =\omega(x,y)\omega(y,x)[ x,y].$

The class of Quasi-Lie algebras incorporates as special cases
\emph{hom-Lie algebras} and more general \emph{quasi-hom-Lie
algebras (qhl-algebras)} which appear naturally in the algebraic
study of $\sigma$-derivations (see \cite{HLS}) and related
deformations of infinite-dimensional and finite-dimensional Lie
algebras. To get the class of qhl-algebras one specifies
$\theta=\omega$ and restricts attention to maps $\alpha$ and $\beta$
satisfying the twisting condition $[
\alpha(x),\alpha(y)]=\beta\circ\alpha [ x,y]$. Specifying this
further by taking $D_\omega =V \times V$, $\beta=id$ and
$\omega=-id$, one gets the class of Hom-Lie algebras including Lie
algebras when $\alpha = id$. The class of quasi-Lie algebras
contains also color Lie algebras and in particular Lie
superalgebras.{\ding{167}}

\section{Hom-Lie admissible algebras and $G$-Hom-associative algebras}
The Lie-admissible algebras was introduced by A. A. Albert in 1948.
Physicists attempted to introduce this structure instead of Lie
algebras. For instance, the validity of Lie-Admissible algebras for
free particles is well known. These algebras arise also in classical
quantum mechanics as a generalization of conventional mechanics (see
\cite{Albert48,MyungMonograph}).

\subsection{Hom-Lie admissible algebras}In
this section, we discuss the concept of Hom-Lie-Admissible algebra,
extending to Hom-algebra setting, the classical notion of Lie-admissible algebra.
\begin{definition}[\cite{MS}]
Let $\A=(V,\mu,\alpha )$ be a Hom-algebra structure.  Then $\A$ is
said to be Hom-Lie-admissible algebra if the bracket
defined for all $x,y \in V$ by
$$
[ x,y ]=\mu (x,y)-\mu (y,x )
$$
satisfies the Hom-Jacobi identity \eqref{HomJacobi}.
\end{definition}
\begin{remark}
Since the bracket  is also skewsymmetric then it defines a Hom-Lie
algebra.
\end{remark}
\begin{remark}
Note that any Hom-associative and  Hom-Lie algebras are
Hom-Lie-admissible.
\end{remark}

We aim now to give another characterization of Hom-Lie admissible
algebras. Let $\A =(V, \mu, \alpha)$ be a Hom-algebra. We denote by
brackets the multiplication of plus Hom-algebra $\A^+$, that is
$[x,y]= \mu(x,y)-\mu(y,x) $ and by $S$ the trilinear map
defined for $x,y,z\in V$ by
$$  S(x,y,z):= \mathfrak{as_\alpha}(x,y,z)+\mathfrak{as_\alpha}(y,z,x)+
\mathfrak{as_\alpha}(z,x,y)
$$
We have the following properties
\begin{lemma}
$$ S(x,y,z)=[\mu(x,y),\alpha(z)]+ [\mu(y,z),\alpha(x)]+[\mu(z,x),\alpha(y)]
$$

\end{lemma}
\begin{proof}
$$[\mu(x,y),\alpha(z)]+ [\mu(y,z),\alpha(x)]+[\mu(z,x),\alpha(y)]=$$
$$\mu(\mu(x,y),\alpha(z))-\mu(\alpha(z),\mu(x,y))+\mu(\mu(y,z),\alpha(x))-\mu(\alpha(x),\mu(y,z))+$$
$$\mu(\mu(z,x),\alpha(y))-\mu(\alpha(y),\mu(z,x))=S(x,y,z)$$
\end{proof}
\begin{proposition}
A Hom-algebra $\A$ is Hom-Lie-admissible if and only if it satisfies
for any $x,y,z$ in $V$
$$ S(x,y,z)=S(x,z,y)$$

\end{proposition}
\begin{proof}
$$S(x,y,z)-S(x,z,y)=$$ $$[\mu(x,y),\alpha(z)]+
[\mu(y,z),\alpha(x)]+[\mu(z,x),\alpha(y)]$$
$$-[\mu(x,z),\alpha(y)]- [\mu(z,y),\alpha(x)]-[\mu(y,x),\alpha(z)]=$$
$$- \circlearrowleft_{x,y,z}[\alpha(x),[y,z]]$$
\end{proof}
\subsection{$G$-Hom-associative algebras}
In the following, we explore some other Hom-Lie-Admissible algebras, extending to Hom-algebra setting
 the results obtained in \cite{GR04}.

\begin{definition}[\cite{MS}]
Let $G$ be a subgroup of the permutations group $\mathcal{S}_3$, a
 Hom-algebra $(V,\mu,\alpha)$ is called $G$-Hom-associative if for any
 $x,y,z\in V$, we have
\begin{equation}\label{admi}
\sum_{\sigma\in G}{(-1)^{\varepsilon ({\sigma})}\mu (\mu (x_{\sigma
(1)},x_{\sigma (2)}),\alpha (x_{\sigma (3)}))}-\mu(\alpha(x_{\sigma
(1)}),\mu (x_{\sigma (2)},x_{\sigma (3)}))=0
\end{equation}
where $x_i$ are in $V$ and $(-1)^{\varepsilon ({\sigma})}$ is the
signature of the permutation $\sigma$.
\end{definition}
The condition (\ref{admi}) may be written in terms of Hom-associator
by
\begin{equation}
\sum_{\sigma\in G}{(-1)^{\varepsilon
({\sigma})}\mathfrak{as_\alpha}\circ \sigma}=0
\end{equation}
where $\sigma$ is the extension of the permutation, still denoted by
the same notation, to a trilinear map defined by
$\sigma(x_1,x_2,x_3)=(x_{\sigma(1)},x_{\sigma(2)},x_{\sigma(3)})$.

\begin{remark}
If $\mu$ is the multiplication of a Hom-Lie-admissible Lie algebra
then the condition \eqref{admi} is equivalent to the property that
the bracket defined by
$$
[ x,y ]=\mu (x,y)-\mu (y,x )
$$
satisfies the Hom-Jacobi condition or equivalently to
\begin{equation}
\sum_{\sigma\in \mathcal{S}_3}{(-1)^{\varepsilon ({\sigma})}\mu (\mu
(x_{\sigma (1)},x_{\sigma (2)}),\alpha (x_{\sigma
(3)}))}-\mu(\alpha(x_{\sigma (1)}),\mu (x_{\sigma (2)},x_{\sigma
(3)}))=0
\end{equation}
which may be written as
\begin{equation}
\sum_{\sigma\in \mathcal{S}_3}{(-1)^{\varepsilon
({\sigma})}\mathfrak{as_\alpha}\circ \sigma}=0.
\end{equation}
\end{remark}

\begin{theorem}[\cite{MS}]
Let $G$ be a subgroup of the permutations group $\mathcal{S}_3$.
Then any $G$-Hom-associative algebra is a Hom-Lie-admissible
algebra.
\end{theorem}
\begin{proof}The skewsymmetry follows straightaway from the
definition.

We have a subgroup $G$ in $\mathcal{S}_3$. Take the set of conjugacy
class $\{g G\}_{g\in I}$  where $I\subseteq G$, and for any
$\sigma_1, \sigma_2\in I,\sigma_1 \neq \sigma_2 \Rightarrow \sigma_1
G\bigcap \sigma_1 G =\emptyset$. Then

$$\sum_{\sigma\in \mathcal{S}_3}{(-1)^{\varepsilon ({\sigma})}
\mathfrak{as_\alpha}\circ \sigma}=\sum_{\sigma_1\in
I}{\sum_{\sigma_2\in \sigma_1 G}{(-1)^{\varepsilon
({\sigma_2})}\mathfrak{as_\alpha}\circ \sigma_2}}=0.$$
\end{proof}

 The result says that for any
 subgroup of $\mathcal{S}_3$ corresponds a class of $G$-Hom-associative
  algebra.  Since the subgroups are
$G_1=\{Id\}, ~G_2=\{Id,\sigma_{1 2}\},~G_3=\{Id,\sigma_{2
3}\},$$ $$~G_4=\{Id,\sigma_{1 3}\},~G_5=A_3 ,~G_6=\mathcal{S}_3,$
where $A_3$ is the alternating group and where $\sigma_{ij}$ is the
transposition between $i$ and $j$, then we obtain the following type
 of Hom-Lie-admissible algebras.
\begin{itemize}
\item The  $G_1$-Hom-associative algebras  are the Hom-associative
algebras defined above.

\item The  $G_2$-Hom-associative algebras satisfy the condition
\begin{equation}\nonumber
\mu(\alpha(x),\mu (y,z))-\mu(\alpha(y),\mu (x,z))=\mu (\mu
(x,y),\alpha (z))-\mu (\mu (y,x),\alpha (z))
\end{equation}
When $\alpha$ is the identity the algebra is called Vinberg algebra
or left symmetric algebra.
\item The  $G_3$-Hom-associative algebras satisfy the condition
\begin{equation}\nonumber
\mu(\alpha(x),\mu (y,z))-\mu(\alpha(x),\mu (z,y))=\mu (\mu
(x,y),\alpha (z))-\mu (\mu (x,z),\alpha (y))
\end{equation}
When $\alpha$ is the identity the algebra is called pre-Lie algebra
or right symmetric algebra.
\item The  $G_4$-Hom-associative algebras satisfy the condition
\begin{equation}\nonumber
\mu(\alpha(x),\mu (y,z))-\mu(\alpha(z),\mu (y,x))=\mu (\mu
(x,y),\alpha (z))-\mu (\mu (z,y),\alpha (x))
\end{equation}
\item The  $G_5$-Hom-associative algebras satisfy the condition
\begin{eqnarray}\nonumber
\nonumber \mu(\alpha(x),\mu (y,z))+\mu(\alpha(y),\mu
(z,x)+\mu(\alpha(z),\mu
(x,y))= \\
\nonumber \mu (\mu (x,y),\alpha (z))+\mu (\mu (y,z),\alpha (x))+\mu
(\mu (z,x),\alpha (y))
\end{eqnarray}\nonumber
If the product $\mu$ is skewsymmetric then the previous condition is
exactly the Hom-Jacobi identity.
\item The  $G_6$-Hom-associative algebras are the Hom-Lie-admissible
algebras.
\end{itemize}
Special cases of $G$-Hom-associative algebras include generalization of Vinberg and
pre-Lie algebras.
\begin{definition}[\cite{MS}]

A Hom-Vinberg algebra (Hom-left-symmetric algebra) is a triple $(V, \mu, \alpha)$ consisting of a
linear space $V$, a bilinear map $\mu: V\times V \rightarrow V$ and
a homomorphism $\alpha$ satisfying
\begin{equation}
\mu(\alpha(x),\mu (y,z))-\mu(\alpha(y),\mu (x,z))=\mu (\mu
(x,y),\alpha (z))-\mu (\mu (y,x),\alpha (z))
\end{equation}
\end{definition}

\begin{definition}[\cite{MS}]
A Hom-pre-Lie  algebra (Hom-right-symmetric algebra) is a triple $(V, \mu, \alpha)$ consisting of
a linear space $V$, a bilinear map $\mu: V\times V \rightarrow V$
and a homomorphism $\alpha$ satisfying
\begin{equation}
\mu(\alpha(x),\mu (y,z))-\mu(\alpha(x),\mu (z,y))=\mu (\mu
(x,y),\alpha (z))-\mu (\mu (x,z),\alpha (y))
\end{equation}
\end{definition}

\begin{remark}
A Hom-pre-Lie algebra is the opposite algebra of a Hom-Vinberg
algebra.
\end{remark}
\begin{remark}
The multiplicative Hom-Novikov algebras which are multiplicative Hom-Vinberg algebra with the additional identity
$\mu(\mu(x,y),\alpha(z))=\mu(\mu(x,z),\alpha(y))$,
were studied in \cite{Yau:HomNovikov}.
\end{remark}

The following theorem states that $G$-associative algebras deform into $G$-Hom-associative
algebras along any algebra endomorphism. Therefore, it provides a construction procedure.
\begin{theorem}[\cite{Yau:homology}]\label{thmConstrGHomAss}
Let $(V,\mu)$ be a $G$-associative algebra  and  $\alpha :
V\rightarrow V$ be an
 algebra endomorphism. Then $(V,\mu_\alpha,\alpha)$,
where $\mu_\alpha=\alpha\circ\mu$,  is a $G$-Hom-associative
algebra.

Moreover, suppose that  $(V',\mu')$ is another $G$-associative
algebra   and $\alpha ' : V'\rightarrow V'$ is an algebra
endomorphism. If $f:V\rightarrow V'$ is an algebras morphism that
satisfies $f\circ\alpha=\alpha'\circ f$ then
$$f:(V,\mu_\alpha,\alpha)\longrightarrow (V',\mu'_{\alpha '},\alpha ')
$$
is a morphism of $G$-Hom-associative algebras.
\end{theorem}

\subsection{Hom-Lie-Admissible Superalgebras}   We discuss in this
section the concept of Hom-Lie-Admissible superalgebras studied in
\cite{AmmarMakhlouf2009}. This study borders also an extension to
graded case of the Lie-admissible algebras discussed in \cite{GR04}.

Let $\A =(V, \mu, \alpha)$ be a Hom-superalgebra, that is a
superspace $V$ with an even bilinear map $\mu$ and an even  linear
map $\alpha$ satisfying eventually identities. Let $[x,y]=
\mu(x,y)-(-1)^{|x||y|}\mu(y,x) $, for all homogeneous element $x,y
\in V$, be the associated supercommutator. The bracket is extended to all elements by
linearity.
\begin{definition}[\cite{AmmarMakhlouf2009}] \label{def:Lieadmis}
Let $\A=(V, \mu, \alpha)$ be a Hom-superalgebra  on $V$ defined by
an even multiplication $\mu$ and an even homomorphism $\alpha$. Then
$\A$ is said to be Hom-Lie admissible superalgebra if the
bracket defined for all homogeneous element $x,y \in V$ by
\begin{equation} \label{commutator}
[ x,y ]=\mu (x,y)-(-1)^{|x||y|}\mu (y,x )
\end{equation}
 satisfies the
Hom-superJacobi identity (\ref{JacobyHomsuperLie}).
\end{definition}

\begin{remark}
Since the supercommutator bracket \eqref{commutator} is always
supersymmetric, this makes any Hom-Lie admissible superalgebra into
a Hom-Lie superalgebra.
\end{remark}
\begin{remark}
 As mentioned  in in the proposition (\ref{Supercommutator}), any
 associative superalgebra is a Hom-Lie admissible superalgebra.
\end{remark}
\begin{lemma}\label{JacobiSuperCommutator}
Let $\A=(V, \mu, \alpha)$ be a Hom-superalgebra and $[ \cdot,\cdot
]$ be the associated supercommutator then
\begin{eqnarray}\label{LieAdmiIdentity}
&&\circlearrowleft_{x,y,z}{(-1)^{|x||z|}[\alpha (x),[y,z]]}=\\
&&\quad (-1)^{|x||z|} \mathfrak{as}_{\alpha}(x,y,z) +
(-1)^{|y||x|}\mathfrak{as}_{\alpha}(y,z,x)\nonumber\\
&&\quad +(-1)^{|z||y|} \mathfrak{as}_{\alpha}(z,x,y)
-(-1)^{|x||z|+|y||z|}\mathfrak{as}_{\alpha}(x,z,y)\nonumber\\
&&\quad -(-1)^{|x||y|+|y||z|}
\mathfrak{as}_{\alpha}(z,y,x)-(-1)^{|x||y|+|x||z|}
\mathfrak{as}_{\alpha}(y,x,z)\nonumber
\end{eqnarray}

\end{lemma}
\begin{proof}
By straightforward calculation.

%we have
%\begin{eqnarray*}
%&&\circlearrowleft_{x,y,z}(-1)^{|x||z|}[\alpha (x),[y,z]]
%=\circlearrowleft_{x,y,z}(-1)^{|x||z|}[\alpha (x),\mu
%(y,z)-(-1)^{|y||z|}\mu (z,y)]
%\\
%&& =(-1)^{|x||z|}\mu (\alpha (x),\mu (y,z))-(-1)^{|x||z|}\mu (\mu (y,z),\alpha (x))\\
%&&\quad +(-1)^{|y||x|}\mu (\alpha (y),\mu (z,x))-(-1)^{|y||z|}\mu (\mu (z,x),\alpha (y))\\
%&&\quad +(-1)^{|z||y|}\mu (\alpha (z),\mu (x,y))-(-1)^{|z||y|}\mu (\mu (x,y),\alpha (z))\\
%&&\quad -(-1)^{|x||z|+|y||z|}\mu (\alpha (x),\mu (z,y))+
%(-1)^{|x||z|+|y||z|}\mu (\mu (z,y),\alpha (x))\\
%&&\quad -(-1)^{|y||x|+|z||x|}\mu (\alpha (y),\mu (x,z))+
%(-1)^{|z||x|+|y||z|}\mu (\mu (x,z),\alpha (y))\\
%&&\quad -(-1)^{|z||y|+|x||y|}\mu (\alpha (z),\mu (y,x))+
%(-1)^{|x||y|+|z||x|}\mu (\mu (y,x),\alpha (z))
%\\
%&&= (-1)^{|x||z|} \mathfrak{as}_{\alpha}(x,y,z)+
%(-1)^{|y||x|}\mathfrak{as}_{\alpha}(y,z,x)\\
%&&\quad +(-1)^{|z||y|} \mathfrak{as}_{\alpha}(z,x,y)
%-(-1)^{|x||z|+|y||z|}\mathfrak{as}_{\alpha}(x,z,y)\\
%&&\quad -(-1)^{|x||y|+|y||z|}
%\mathfrak{as}_{\alpha}(z,y,x)-(-1)^{|x||y|+|x||z|}
%\mathfrak{as}_{\alpha}(y,x,z)
%\end{eqnarray*}
\end{proof}
\begin{remark}
If $\alpha =\id$, then we obtain a formula expressing the left hand
side of the classical superJacobi identity in terms of classical
associator.
\end{remark}
 In the following we aim to characterize the Hom-Lie admissible superalgebras
 in terms of Hom-associator.
 We introduce a trilinear map $\widetilde{S}$ defined for homogeneous elements
  $x,y,z\in V$ by
$$\widetilde{S}(x,y,z):= (-1)^{|x||z|}\mathfrak{as}_{\alpha}(x,y,z)+
(-1)^{|y||x|}\mathfrak{as}_{\alpha}(y,z,x)+(-1)^{|z||y|}
\mathfrak{as}_{\alpha}(z,x,y).
$$
Then, we have the following properties.

\begin{proposition}[\cite{AmmarMakhlouf2009}]
Let  $\A =(V, \mu, \alpha)$ be a Hom-superalgebra, then $\A$ is a
Hom-Lie admissible superalgebra if and only if it satisfies
\begin{equation}\label{Srel} \widetilde{S}(x,y,z)=
(-1)^{|x||y|+|x||z|+|y||z|}\widetilde{S}(x,z,y)
\end{equation}
for any homogenous elements  $x,y,z \in V.$
\end{proposition}
\begin{proof} We have
\begin{eqnarray*}
&& \widetilde{S}(x,y,z)-(-1)^{|x||y|+|x||z|+|y||z|}S(x,z,y)
\\&&=(-1)^{|x||z|}
\mathfrak{as}_{\alpha}(x,y,z)+
(-1)^{|y||x|}\mathfrak{as}_{\alpha}(y,z,x)+(-1)^{|z||y|}
\mathfrak{as}_{\alpha}(z,x,y)\\
&& \quad
-(-1)^{|x||y|+|x||z|+|y||z|}((-1)^{|x||y|}\mathfrak{as}_{\alpha}(x,z,y)+
(-1)^{|z||x|}\mathfrak{as}_{\alpha}(z,y,x)\\
&&\quad+(-1)^{|y||z|} \mathfrak{as}_{\alpha}(y,x,z))
\\
&&=(-1)^{|x||z|} \mathfrak{as}_{\alpha}(x,y,z)+
(-1)^{|y||x|}\mathfrak{as}_{\alpha}(y,z,x)+(-1)^{|z||y|}
\mathfrak{as}_{\alpha}(z,x,y)\\
&&\quad
-(-1)^{|x||z|+|y||z|}\mathfrak{as}_{\alpha}(x,z,y)-(-1)^{|x||y|+|y||z|}
\mathfrak{as}_{\alpha}(z,y,x)\\
&&\quad-(-1)^{|x||y|+|x||z|}
\mathfrak{as}_{\alpha}(y,x,z)\\
 &&= \circlearrowleft_{x,y,z}{(-1)^{|x||z|}[x,[y,z]]} \quad
 \text{(lemma \ref{JacobiSuperCommutator})}
\end{eqnarray*}
Then the Hom-superJacobi identity (\ref{JacobyHomsuperLie}) is
satisfied if and only if the condition (\ref{Srel}) holds.
\end{proof}

\subsection{$G$-Hom-associative superalgebras}
A classification of Hom-Lie admissible superalgebras
 using the symmetric group $\mathcal{S}_3$ is provided in \cite{AmmarMakhlouf2009}.
 We extended to $\mathbb{Z}_2$-graded case the
  notion of $G$-Hom-associative algebras and in particular  $G$-associative algebras
 which was introduced in the classical ungraded Lie case in (\cite{GR04}) and
  developed for the Hom-Lie case in (\cite{MS}).

Let $\mathcal{S}_3$ be the permutation group generated by the transpositions $\sigma_1
,\sigma_2.$ We extend a permutation $\tau \in \mathcal{S}_3$ to a
map $\tau : V^{\times 3} \rightarrow V^{\times 3}$ defined for
$x_1,x_2,x_3\in V$ by
$\tau(x_1,x_2,x_3)=(x_{\tau(1)},x_{\tau(2)},x_{\tau(3)}).$ We keep
for simplicity
 the same notation. In particular,
 $\sigma_1(x_1,x_2,x_3)=(x_{2},x_{1},x_{3})$ and
 $\sigma_2(x_1,x_2,x_3)=(x_{1},x_{3},x_{2}).$

 We introduce a notion of a parity of transposition $\sigma_i$ where $i\in \{1,2\}$, by setting
   $$|\sigma_i(x_1,x_2,x_3)|=|x_i||x_{i+1}|.$$
 We assume that the parity of the identity is $0$ and for the composition
 $\sigma_i\sigma_j$, it is defined by
\begin{eqnarray*}|\sigma_i\sigma_j(x_1,x_2,x_3)|&=&
 |\sigma_j(x_1,x_2,x_3)|+|\sigma_i(\sigma_j(x_1,x_2,x_3))|\\
 \ &=& |\sigma_j(x_1,x_2,x_3)|+
 |\sigma_i(x_{\sigma _j(1)},x_{\sigma_ j(2)},x_{\sigma_j(3)})|
 \end{eqnarray*}
We define by induction the parity for any composition.

For the elements $\id,\sigma_1 ,\sigma_2,\sigma_1 \sigma_2,\sigma_2
\sigma_1, \sigma_2 \sigma_1 \sigma_2$ of $\mathcal{S}_3$, we obtain
\begin{eqnarray*}
\mid \id (x_1,x_2,x_3)\mid&=&0,\\
\mid\sigma_1(x_1,x_2,x_3)\mid &=&|x_1||x_{2}|
 ,\\
 \mid\sigma_2(x_1,x_2,x_3)\mid &=&|x_2||x_{3}|,\\
 \mid \sigma_1 \sigma_2(x_1,x_2,x_3)\mid &=&|x_2||x_{3}|+|x_1||x_{3}|
 ,\\
  \mid \sigma_2 \sigma_1(x_1,x_2,x_3)\mid &=&|x_1||x_{2}|+|x_1||x_{3}|,\\
\mid \sigma_2 \sigma_1 \sigma_2(x_1,x_2,x_3)\mid
&=&|x_2||x_{3}|+|x_1||x_{3}|+ |x_1||x_{2}|.
\end{eqnarray*}

Now, we  express the condition of Hom-Lie admissibility of a
Hom-superalgebra using  permutations.

\begin{lemma}[\cite{AmmarMakhlouf2009}]\label{JacobiSuperCommutator2}
A Hom-superalgebra $\A=(V, \mu, \alpha)$ is a Hom-Lie admissible
superalgebra if the following condition holds
\begin{equation}\label{LieAdmiIdentity2}
\sum_{\tau\in \mathcal{S}_3} {(-1)^{\varepsilon ({\tau})}
(-1)^{|\tau(x_1,x_2,x_3)|}\mathfrak{as}_{\alpha}\circ
\tau(x_1,x_2,x_3)} =0
\end{equation}
where $x_i$ are in $V$, $(-1)^{\varepsilon ({\tau})}$ is the
signature of the permutation $\tau$ and $|\tau(x_1,x_2,x_3)|$ is the
parity of $\tau$.
\end{lemma}
\begin{proof}
By straightforward calculation, the associated supercommutator
bracket satisfies
\begin{eqnarray*}
&& \circlearrowleft_{x_1,x_2,x_3}(-1)^{|x_1||x_3|}[\alpha
(x_1),[x_2,x_3]] =\\ && (-1)^{|x_1||x_3|}\sum_{\tau\in
\mathcal{S}_3} {(-1)^{\varepsilon ({\tau})}
(-1)^{|\tau(x_1,x_2,x_3)|}\mathfrak{as}_{\alpha}\circ
\tau(x_1,x_2,x_3)}.
\end{eqnarray*}
\end{proof}
It turns out that, for the associated supercommutator   of a
Hom-superalgebra,
  the Hom-superJacobi identity (\ref{JacobyHomsuperLie}) is equivalent to
\begin{equation*}
\sum_{\tau\in \mathcal{S}_3} {(-1)^{\varepsilon ({\tau})}
(-1)^{|\tau(x_1,x_2,x_3)|}\mathfrak{as}_{\alpha}\circ
\tau(x_1,x_2,x_3)}=0.
\end{equation*}
We introduce now the notion of $G$-Hom-associative superalgebra, where
$G$ is
 a subgroup of the permutations group $\mathcal{S}_3$.
\begin{definition}[\cite{AmmarMakhlouf2009}]
Let $G$ be a subgroup of the permutations group $\mathcal{S}_3$, a
Hom-superalgebra on $V$ defined by the multiplication $\mu$ and a
homomorphism $\alpha$ is said to  be $G$-Hom-associative
superalgebra if
\begin{equation}\label{admiSAL}
\sum_{\tau\in G} {(-1)^{\varepsilon ({\tau})}
(-1)^{|\tau(x_1,x_2,x_3)|}\mathfrak{as}_{\alpha}\circ
\tau(x_1,x_2,x_3)}=0.
\end{equation}
where $x_i$ are in $V$, $(-1)^{\varepsilon ({\tau})}$ is the
signature of the permutation and $|\tau(x_1,x_2,x_3)|$ is the parity
of $\tau$ defined above.
\end{definition}
In particular, we call $G$-associative
superalgebra a $G$-Hom-associative
superalgebra where $\alpha$ is the identity map.

The following result is a graded version  of the results obtained in
(\cite{GR04,MS}).

\begin{theorem}[\cite{AmmarMakhlouf2009}]
Let $G$ be a subgroup of the permutations group $\mathcal{S}_3$.
Then any $G$-Hom-associative superalgebra is a Hom-Lie admissible
superalgebra.
\end{theorem}
\begin{proof}The supersymmetry follows straightaway from the
definition.

We have a subgroup $G$ in $\mathcal{S}_3$. Take the set of conjugacy
class $\{g G\}_{g\in I}$  where $I\subseteq G$, and for any $\tau_1,
\tau_2\in I,\sigma_1 \neq \tau_2 \Rightarrow \tau_1 G\bigcap \tau_2
G =\emptyset$. Then

\begin{eqnarray*}&& \sum_{\tau\in \mathcal{S}_3}{(-1)^{\varepsilon ({\tau})}(-1)^{|\tau(x_1,x_2,x_3)|}
\mathfrak{as}_{\alpha}\circ \tau}(x_1,x_2,x_3)=\\&& \sum_{\tau_1\in
I}{\sum_{\tau_2\in \tau_1 G}{(-1)^{\varepsilon
({\tau_2})}(-1)^{|\tau_2(x_1,x_2,x_3)|}\mathfrak{as}_ {\alpha}\circ
\tau_2}}(x_1,x_2,x_3)=0
\end{eqnarray*}
where $(x_1,x_2,x_3)\in V$, with $V$  the underlaying superspace of
the $G$-Hom-associative superalgebra.
\end{proof}
It follows that in particular, we have:
\begin{corollary}
Let $G$ be a subgroup of the permutations group $\mathcal{S}_3$.
Then any $G$-associative superalgebra is a Lie admissible
superalgebra.
\end{corollary}
Now, we provide a classification of the Hom-Lie admissible
superalgebras through $G$-Hom-associative superalgebras.

 The subgroups of $\mathcal{S}_3$, which are
$G_1=\{Id\}, ~G_2=\{Id,\sigma_{1}\},~G_3=\{Id,\sigma_{2}\},
~G_4=\{Id,\sigma_2\sigma_1\sigma_2\},~G_5=A_3 ,~G_6=\mathcal{S}_3,$
where $A_3$ is the alternating group, lead to  the following six classes
 of Hom-Lie admissible superalgebras.
\begin{itemize}
\item The  $G_1$-Hom-associative superalgebras  are the Hom-associative
superalgebras defined above.

\item The  $G_2$-Hom-associative superalgebras satisfy the identity
\begin{eqnarray*}
&&\mu(\alpha(x),\mu (y,z))- \mu (\mu (x,y),\alpha (z))=\\&&
\quad\quad \quad (-1)^{|x||y|}(\mu(\alpha(y),\mu (x,z))-\mu (\mu
(y,x),\alpha (z)))
\end{eqnarray*}
\item The  $G_3$-Hom-associative superalgebras satisfy the identity
\begin{eqnarray*}
&&\mu(\alpha(x),\mu (y,z))-\mu (\mu (x,y),\alpha (z))=\\&&
\quad\quad \quad (-1)^{|y||z|}(\mu(\alpha(x),\mu (z,y))-\mu (\mu
(x,z),\alpha (y)))
\end{eqnarray*}
\item The  $G_4$-Hom-associative superalgebras satisfy the identity
\begin{eqnarray*}
&&\mu(\alpha(x),\mu (y,z))-\mu (\mu (x,y),\alpha (z))=\\&&
\quad\quad \quad (-1)^{|x||y|+|x||z|+|y||z|}(\mu(\alpha(z),\mu
(y,x))-\mu (\mu (z,y),\alpha (x)))
\end{eqnarray*}
\item The  $G_5$-Hom-associative superalgebras satisfy the identity
\begin{eqnarray*}
&& -(-1)^{|x||z|+|y||z|}(\mu(\alpha(z),\mu (x,y)) +\mu (\mu
(z,x),\alpha (y)))= \\&& \quad\quad \quad \mu(\alpha(x),\mu
(y,z))-\mu (\mu (x,y),\alpha (z))\\&& \quad\quad \quad
+(-1)^{|x||y|+|x||z|}(\mu(\alpha(y),\mu (z,x)+\mu (\mu (y,z),\alpha
(x)))
\end{eqnarray*}
\item The  $G_6$-Hom-associative superalgebras are the Hom-Lie admissible
superalgebras.
\end{itemize}

\begin{remark}
Moreover, if in the previous identities we consider $\alpha=\id $,
then we obtain a classification of Lie-admissible superalgebras.
\end{remark}
\begin{remark}A  $G_2$-Hom-associative (resp. $G_2$-associative)
 superalgebras  might be called \emph{Hom-Vinberg superalgebras}
 (resp. \emph{Vinberg superalgebras}) or \emph{Hom-left-symmetric superalgebras}
 (resp. left-symmetric superalgebras).

Similarly, A  $G_3$-Hom-associative (resp. $G_3$-associative) superalgebras might be called
\emph{Hom-pre-Lie superalgebras} (resp. \emph{pre-Lie
superalgebras}) or \emph{Hom-right-symmetric superalgebras}
 (resp. right-symmetric superalgebras).

 Notice that a Hom-pre-Lie superalgebra is the
opposite algebra of a Hom-Vinberg superalgebra. Therefore, they
actually form a same class.
\end{remark}

The following result generalizes the theorem \ref{thm:SALmorphism} to  any
$G$-Hom-associative superalgebra.
\begin{theorem}
\label{thm:GAssSALmorphism} Let $(V,[\cdot,\cdot])$ be a
$G$-associative superalgebra  and  $\alpha : V\rightarrow V$ be an
even  $G$-Hom-associative superalgebra endomorphism. Then
$(V,[\cdot,\cdot]_\alpha,\alpha)$, where
$[x,y]_\alpha=\alpha([x,y])$,  is a $G$-Hom-associative
superalgebra.

Moreover, suppose that  $(V',[\cdot,\cdot]')$ is another
$G$-associative superalgebra and  $\alpha ' : V'\rightarrow V'$ is a
$G$-associative superalgebra endomorphism. If $f:V\rightarrow V'$ is
a $G$-associative superalgebra morphism that satisfies
$f\circ\alpha=\alpha'\circ f$ then
$$f:(V,[\cdot,\cdot]_\alpha,\alpha)\longrightarrow (V',[\cdot,\cdot]',\alpha ')
$$
is a morphism of $G$-Hom-associative superalgebras.
\end{theorem}
\begin{proof}Similar to proof of theorem \ref{thm:SALmorphism}.
\end{proof}
{\ding{167}}
\section{Hom-alternative algebras}
Now, we introduce and discuss the notion of Hom-alternative algebra,
following \cite{Mak:HomAlternativeHomJord}.
\begin{definition}[\cite{Mak:HomAlternativeHomJord}]
A \emph{left Hom-alternative} algebra (resp. \emph{right
Hom-alternative algebra}) is a triple $(V,\mu,\alpha )$ consisting
of a $\K$-linear space $V$, a linear map $\alpha: V \rightarrow V$
and a multiplication $\mu: V\otimes V \rightarrow V$ satisfying, for any $x,y$ in $V$, the
left Hom-alternative identity, that is
\begin{equation}\label{HomLeftAlternative}
\mu (\alpha (x),\mu(x, y))=\mu(\mu(x, x) , \alpha (y)),
\end{equation}
respectively, right Hom-alternative identity, that is
\begin{equation}\label{HomRightAlternative}
\mu (\alpha(x),\mu(y, y))=\mu(\mu(x,y) ,\alpha( y)).
\end{equation}

A Hom-alternative algebra is one which is both left and right
Hom-alternative algebra.
\end{definition}

\begin{remark}
Any Hom associative algebra is a Hom-alternative algebra.
\end{remark}

Using the Hom-associator \eqref{HomAssociator}, the condition
\eqref{HomLeftAlternative} (resp. \eqref{HomRightAlternative}) may
be written using Hom-associator respectively
\begin{eqnarray*}
 \mathfrak{as}_\alpha(x,x,y)=0 ,\quad \mathfrak{as}_\alpha(y,x,x)=0.
 \end{eqnarray*}

By linearization, we have the following equivalent definition of
left and right Hom-alternative algebras.

\begin{proposition}
A triple  $(V,\mu,\alpha)$ is a left Hom-alternative algebra (resp.
right alternative algebra) if and only if the identity
\begin{equation}\label{HomLeftAlternativeLineariz}
\mu (\alpha (x),\mu(y, z))-\mu(\mu(x, y),\alpha( z))+\mu (\alpha
(y), \mu(x, z))-\mu(\mu(y, x), \alpha (z))=0.
\end{equation}
respectively,
\begin{equation}\label{HomRightAlternativeLineariz}
\mu (\alpha(x),\mu (y, z))-\mu (\mu (x, y), \alpha(z))+\mu
(\alpha(x), \mu (z, y))-\mu (\mu (x, z), \alpha(y))=0.
\end{equation}
holds.
\end{proposition}
\begin{proof}
We assume that,  for any $x,y,z\in V$,
$\mathfrak{as}_\alpha(x,x,z)=0 $ (left alternativity), then we
expand $\mathfrak{as}_\alpha(x+y,x+y,z)=0 $.

The proof for right Hom-alternativity is obtained by expanding
$\mathfrak{as}_\alpha(x,y+z,y+z)=0 .$

Conversely, we set $x=y$ in (\ref{HomLeftAlternativeLineariz}), respectively $y=z$ in (\ref{HomRightAlternativeLineariz}).
\end{proof}

  \begin{remark} The multiplication could be considered as a linear map
  $\mu : V \otimes V \rightarrow V$, then the condition
  (\ref{HomLeftAlternativeLineariz}) (resp. (\ref{HomRightAlternativeLineariz})) writes
  \begin{equation}\label{HomLeftAlternativeLineariz2}
\mu \circ (\alpha\otimes \mu-\mu \otimes \alpha)\circ (id^{\otimes
3} +\sigma_{1})=0,
\end{equation}
respectively
\begin{equation}\label{HomRightAlternativeLineariz2}
\mu \circ (\alpha\otimes \mu-\mu \otimes \alpha)\circ (id^{\otimes
3}+\sigma_{2})=0.
\end{equation}
where $id$ stands for the identity map and $\sigma_{1}$ and
$\sigma_{2}$ stands for trilinear maps defined for any $x_1,x_2, x_3\in V$ by
\begin{eqnarray*}
\sigma_{1}(x_1\otimes x_2\otimes x_3)=x_2\otimes x_1\otimes x_3,\ \
\sigma_{2}(x_1\otimes x_2\otimes x_3)=x_1\otimes x_3\otimes
x_2.
\end{eqnarray*}

In terms of associators, the identities
(\ref{HomLeftAlternativeLineariz}) (resp.
(\ref{HomRightAlternativeLineariz})) are equivalent respectively to
\begin{equation}\label{Alt0}\mathfrak{as}_\alpha+\mathfrak{as}_\alpha\circ\sigma_{1}=0 \quad
\text{ and  }\
\mathfrak{as}_\alpha+\mathfrak{as}_\alpha\circ\sigma_{2}=0.
\end{equation}
  \end{remark}
  Hence, for any $x,y,z\in V$, we have
  \begin{equation}\label{Alt1}
  \mathfrak{as}_\alpha(x,y,z)=-\mathfrak{as}_\alpha(y,x,z)\quad \text{ and
  }\quad \mathfrak{as}_\alpha(x,y,z)=-\mathfrak{as}_\alpha(x,z,y).
  \end{equation}
  We have also the following property.
\begin{lemma}\label{HomAlternating1}
Let $(V,\mu,\alpha)$ be an Hom-alternative algebra. Then
\begin{equation}\label{Alt2}
\mathfrak{as}_\alpha(x,y,z)=-\mathfrak{as}_\alpha(z,y,x).
\end{equation}
\end{lemma}
\begin{proof}
Using (\ref{Alt0}), we have
\begin{eqnarray*}
\mathfrak{as}_\alpha(x,y,z)+\mathfrak{as}_\alpha(z,y,x)=
-\mathfrak{as}_\alpha(y,x,z)-\mathfrak{as}_\alpha(y,z,x)
=0.
\end{eqnarray*}
\end{proof}
\begin{remark}
The identities \eqref{Alt0},\eqref{Alt1} lead to the fact that an
algebra is Hom-alternative if and only if the Hom-associator
$\mathfrak{as}_\alpha(x,y,z)$ is an alternating function of its
arguments, that is
$$\mathfrak{as}_\alpha(x,y,z)=-\mathfrak{as}_\alpha(y,x,z)=
-\mathfrak{as}_\alpha(x,z,y)=-\mathfrak{as}_\alpha(z,y,x).$$
\end{remark}
\begin{proposition}
A Hom-alternative algebra is Hom-flexible.
\end{proposition}
\begin{proof}
Using  lemma \ref{HomAlternating1}, we have
$\mathfrak{as}_\alpha(x,y,x)=-\mathfrak{as}_\alpha(x,y,x)$.

Therefore, $\mathfrak{as}_\alpha(x,y,x)=0$.
\end{proof}

\begin{proposition}
Let $(V,\mu,\alpha)$ be a Hom-alternative algebra and $x,y,z\in V$.

If $x$ and $y$ anticommute, that is $\mu (x,y)=-\mu (y,x)$, then we
have
\begin{equation}
\mu (\alpha (x),\mu (y,z))=-\mu (\alpha (y),\mu (x,z),
\end{equation}
and
\begin{equation}
\mu (\mu (z,x),\alpha (y))=-\mu (\mu (z,y),\alpha (x)).
\end{equation}
\end{proposition}
\begin{proof}
The left alternativity leads to
\begin{equation}
\mu (\alpha (x),\mu(y, z))-\mu(\mu(x, y),\alpha( z))+\mu (\alpha
(y), \mu(x, z))-\mu(\mu(y, x), \alpha (z))=0.
\end{equation}
Since $\mu(x, y)=-\mu(y, x)$, then the previous identity becomes
\begin{equation}
\mu (\alpha (x),\mu(y, z))+\mu (\alpha (y), \mu(x, z))=0.
\end{equation}
Similarly, using the right alternativity and the assumption of
anticommutativity, we get the second identity.
\end{proof}

The following theorem provides a way to construct Hom-alternative
algebras starting from an alternative algebra and an algebra
endomorphism. This procedure was applied to associative algebras,
$G$-associative algebras and Lie algebra in \cite{Yau:homology}. It
was  extended to coalgebras in \cite{HomAlgHomCoalg} and to $n$-ary
algebras of Lie type respectively associative type in
\cite{AMS2009}.

\begin{theorem}[\cite{Mak:HomAlternativeHomJord}]\label{thmConstrHomAlt}
Let $(V,\mu)$ be a left  alternative algebra (resp. a right
alternative algebra ) and  $\alpha : V\rightarrow V$ be an
 algebra endomorphism. Then $(V,\mu_\alpha,\alpha)$,
where $\mu_\alpha=\alpha\circ\mu$,  is a left  Hom-alternative
algebra (resp. right Hom-alternative algebra).

Moreover, suppose that  $(V',\mu')$ is another left  alternative algebra (resp. a right
alternative algebra )  and $\alpha ' : V'\rightarrow V'$ is an algebra
endomorphism. If $f:V\rightarrow V'$ is an algebras morphism that
satisfies $f\circ\alpha=\alpha'\circ f$ then
$$f:(V,\mu_\alpha,\alpha)\longrightarrow (V',\mu'_{\alpha '},\alpha ')
$$
is a morphism of left  Hom-alternative algebras (resp. right Hom-alternative algebras).
\end{theorem}

\begin{remark}\label{HomAltInducedByAlt}
Let  $(V,\mu,\alpha)$ be a Hom-alternative algebra, one may ask whether this
Hom-alternative algebra is   induced by an
ordinary alternative algebra $(V,\widetilde{\mu})$, that is $\alpha$
is an algebra endomorphism with respect to $\widetilde{\mu}$ and
$\mu=\alpha\circ\widetilde{\mu}$. This question was addressed and discussed for
Hom-associative algebras in \cite{FregierGohr2,Gohr}.

First observation, if $\alpha$
is an algebra endomorphism with respect to $\widetilde{\mu}$
then $\alpha$ is also an algebra endomorphism
with respect to $\mu$. Indeed,
$$\mu(\alpha(x),\alpha(y))=\alpha\circ\widetilde{\mu}(\alpha(x),\alpha(y))=
\alpha\circ\alpha\circ\widetilde{\mu}(x,y)=\alpha\circ\mu(x,y).$$

Second observation, if $\alpha$ is bijective then $\alpha^{-1}$ is
also an algebra automorphism. Therefore one may use an untwist
operation on the Hom-alternative algebra in order to recover the
alternative algebra ($\widetilde{\mu}=\alpha^{-1}\circ\mu$).
\end{remark}

\subsection{ Examples of Hom-Alternative algebras}
We construct examples of Hom-alternative using theorem
(\ref{thmConstrHomAlt}).
We use to this end the classification of 4-dimensional alternative algebras
which are not associative (see \cite{EGG}) and the algebra of octonions (see \cite{Baez}).
 For each algebra, algebra endomorphisms are provided. Therefore, Hom-alternative
 algebras are attached according to theorem
(\ref{thmConstrHomAlt}).

\begin{example}
[Hom-alternative algebras of dimension 4]  According to \cite{EGG}, p 144,
there are exactly two alternative but not associative algebras of
dimension 4 over any field.
With respect to a basis $\{e_0, e_1, e_2, e_3\}$, one algebra is given by the
following multiplication (the unspecified products are zeros)
\begin{eqnarray*} && \mu_1(e_0,e_0)=e_0, \; \mu_1 (e_0,e_1)=e_1,
\;\mu_1 (e_2,e_0)=e_2,\\&&
\mu_1 (e_2,e_3)=e_1, \;\mu_1 (e_3,e_0)=e_3, \;\mu_1 (e_3,e_2)=- e_1.
\end{eqnarray*}
The other algebra is given by
\begin{eqnarray*}&& \mu_2 ( e_0,e_0)=e_0, \;\mu_2 ( e_0,e_2)=e_2,
 \;\mu_2 (e_0,e_3)=e_3,
 \\&& \mu_2 (e_1,e_0)=e_1, \;\mu_2 (e_2,e_3)=e_1, \;\mu_2 (e_3,e_2)=- e_1.
 \end{eqnarray*}
These two alternative algebras are anti-isomorphic, that is the first
one is isomorphic to the opposite of the second one. The algebra endomorphisms of $\mu_1$ and  $\mu_2$ are exactly the same.
 We provide two examples of algebra endomorphisms for these algebras.
\begin{enumerate}
 \item The algebra endomorphism  $\alpha_1$ with respect to the same basis is defined by
\begin{eqnarray*}
&& \alpha_1(e_0)= e_0+a_1\ e_1+a_2\ e_2+a_3\ e_3 ,\ \
\alpha_1(e_1)=0,\ \ \\&&
\alpha_1(e_2)=a_4 \ e_2+\frac{a_4 a_3}{a_2} \ e_3,\ \
\alpha_1(e_3)=a_5 \ e_2+\frac{a_5 a_3}{a_2} \ e_3,
\end{eqnarray*}
with $a_1,\cdots,a_5\in \K$ and $a_2\neq 0$.
\item The algebra endomorphism $\alpha_2$ with respect to the same basis is defined by
\begin{eqnarray*}
&& \alpha_2(e_0)= e_0+a_1\ e_1+a_2\ e_2+a_3\ e_3 ,\ \
\alpha_2(e_1)=a_4\ e_1,\ \ \\&&
\alpha_2(e_2)=- \frac{a_4 a_2}{a_5} \ e_2- \frac{a_4 a_3}{a_5} \ e_3,\ \
\alpha_2(e_3)=a_5 \ e_1+a_6 \ e_2+\frac{a_6 a_3-a_5}{a_2} \ e_3,
\end{eqnarray*}
with $a_1,\cdots,a_6\in \K$ and $a_2,a_5\neq 0$.
\end{enumerate}

According to theorem (\ref{thmConstrHomAlt}), the linear map $\alpha_1$ an the
following multiplications

 \begin{align*} & \bullet  \mu^1_1(e_0,e_0)=e_0+a_1\ e_1+a_2\ e_2+a_3\ e_3, \;
\mu^1_1 (e_0,e_1)=0,\\& \ \ \
\mu^1_1 (e_2,e_0)=a_4 \ e_2+\frac{a_4 a_3}{a_2} \ e_3,
\mu^1_1 (e_2,e_3)=0, \\& \ \ \ \mu^1_1 (e_3,e_0)=a_5 \ e_2+\frac{a_5 a_3}{a_2} \ e_3, \;\mu_1 (e_3,e_2)=0.
\end{align*}

\begin{align*}& \bullet \mu^1_2 ( e_0,e_0)=e_0+a_1\ e_1+a_2\ e_2+a_3\ e_3, \;
\mu^1_2 ( e_0,e_2)=a_4 \ e_2+\frac{a_4 a_3}{a_2} \ e_3,
 \;\\& \ \ \ \mu^1_2 (e_0,e_3)=a_5 \ e_2+\frac{a_5 a_3}{a_2} \ e_3,
 \; \mu^1_2 (e_1,e_0)=0, \;\mu^1_2 (e_2,e_3)=0, \;\mu^1_2 (e_3,e_2)=0.
 \end{align*}
 determine  4-dimensional Hom-alternative algebras.

The linear map $\alpha_2$ leads to the following multiplications

\begin{align*}& \bullet \mu^2_1(e_0,e_0)=e_0+a_1\ e_1+a_2\ e_2+a_3\ e_3, \;
 \mu^2_1 (e_0,e_1)=a_4\ e_1,\\& \ \ \
\;\mu^2_1 (e_2,e_0)=- \frac{a_4 a_2}{a_5} \ e_2- \frac{a_4 a_3}{a_5} \ e_3,\\& \ \ \
\mu^2_1 (e_2,e_3)=a_4\ e_1, \;\mu^2_1 (e_3,e_0)=e_3,
\;\mu^2_1 (e_3,e_2)=- a_4\ e_1.
\end{align*}

\begin{align*}& \bullet \mu^2_2 ( e_0,e_0)=e_0+a_1\ e_1+a_2\ e_2+a_3\ e_3,
 \;\mu^2_2 ( e_0,e_2)=- \frac{a_4 a_2}{a_5} \ e_2- \frac{a_4 a_3}{a_5} \ e_3,
 \;\\& \ \ \ \mu^2_2 (e_0,e_3)=a_5 \ e_1+a_6 \ e_2+\frac{a_6 a_3-a_5}{a_2} \ e_3,
 \;\\& \ \ \ \mu^2_2 (e_1,e_0)=a_4\ e_1, \;\mu_2 (e_2,e_3)=a_4\ e_1,
  \;\mu^2_2 (e_3,e_2)=- a_4\ e_1.
 \end{align*}
\end{example}

\begin{example}[Octonions]
{\rm Octonions are typical example of alternative algebra. They were discovered in 1843 by John T. Graves who
called them Octaves and independently by Arthur Cayley in 1845. See \cite{Baez} for
the role of the octonions in algebra, geometry and topology and see
also \cite{Albuquerque} where octonions are viewed as a quasialgebra. The
octonions algebra which is also called Cayley Octaves or Cayley
algebra is an  $8$-dimensional defined with respect to a basis
$\{u,e_1,e_2,e_3,e_4,e_5,e_6,e_7\}$, where $u$ is the identity for the
multiplication, by the following multiplication table.  The table
describes multiplying the $i$th row elements  by the $j$th column
elements.

\[
\begin{array}{|c|c|c|c|c|c|c|c|c|}
  \hline
   \ & u& e_1 & e_2 & e_3 & e_4 & e_5 & e_6 & e_7 \\ \hline
   u& u& e_1 & e_2 & e_3 & e_4 & e_5 & e_6 & e_7 \\ \hline
   e_1 &e_1 & -u & e_4 & e_7 & -e_2 & e_6 & -e_5 &- e_3 \\ \hline
   e_2 &e_2 & -e_4 & -u & e_5 & e_1 & -e_3 & e_7 & -e_6 \\ \hline
   e_3 &e_3 & -e_7 & -e_5 & -u & e_6 & e_2 & -e_4 & e_1 \\ \hline
   e_4 &e_4 & e_2 & -e_1 & -e_6 & -u & e_7 & e_3 & -e_5 \\ \hline
   e_5 &e_5 & -e_6 & e_3 & -e_2 & -e_7 & -u & e_1 & e_4 \\ \hline
   e_6 &e_6 & e_5 & -e_7 & e_4 & -e_3 & -e_1 & -u & e_2 \\ \hline
   e_7 &e_7 & e_3 & e_6 & -e_1 & e_5 & -e_4 & -e_2 & -u \\
    \hline
\end{array}
\]
\

The diagonal algebra endomorphisms of octonions are give by maps $\alpha$
defined with respect to the basis $\{u,e_1,e_2,e_3,e_4,e_5,e_6,e_7\}$ by
\begin{eqnarray*}
&& \alpha(u)=u,\ \ \alpha(e_1)=a\ e_1,\ \ \alpha(e_2)=b\ e_2,\ \
\alpha(e_3)=c \ e_3
,\\&&
 \alpha(e_4)=a b\ e_4,\ \ \alpha(e_5)=b c\  e_5
,\ \ \alpha(e_6)=a b c\  e_6,\ \ \alpha(e_7)=a c\ e_7,
\end{eqnarray*}
where $a,b,c$ are any parameter in $\K$.
The associated Hom-alternative algebra to the octonions algebra according to
theorem (\ref{thmConstrHomAlt}) is described by the map $\alpha$ and the
multiplication   defined by the following table. The table
describes multiplying the $i$th row elements  by the $j$th column
elements.

\[
\begin{array}{|c|c|c|c|c|c|c|c|c|}
  \hline
   \ & u& e_1 & e_2 & e_3 & e_4 & e_5 & e_6 & e_7 \\ \hline
   u& u& a e_1 & b e_2 &c e_3 & ab e_4 & bc e_5 & a bc e_6 & a c e_7 \\ \hline
   e_1 &a e_1 & -u & ab e_4 & a c e_7 & -b e_2 & a bc e_6 & -bc e_5 &- c e_3 \\ \hline
   e_2 &b e_2 & -ab e_4 & -u & bc e_5 & a e_1 & -c e_3 & a c e_7 & -a bc e_6 \\ \hline
   e_3 &c e_3 & -a c e_7 & -bc e_5 & -u & a bc e_6 & b e_2 & -ab e_4 & a e_1 \\ \hline
   e_4 &ab e_4 & b e_2 & -a e_1 & -a bc e_6 & -u & a c e_7 & c e_3 & -bc e_5 \\ \hline
   e_5 &bc e_5 & -a bc e_6 & c e_3 & -b e_2 & -a c e_7 & -u & a e_1 & ab e_4 \\ \hline
   e_6 &a bc e_6 & bc e_5 & -a c e_7 & ab e_4 & -c e_3 & -a e_1 & -u & b e_2 \\ \hline
   e_7 &a c e_7 & c e_3 & a bc e_6 & -a e_1 & bc e_5 & -ab e_4 & -b e_2 & -u \\
    \hline
\end{array}
\]
Notice that the new algebra is no longer unital, neither  an
alternative algebra since
$$\mu(u,\mu(u,e_1))-\mu(\mu(u,u),e_1)=(a^2-a)e_1,
$$
which is different from $0$ when $a\neq 0,1$.
}
\end{example}

{\ding{167}}

\section{Hom-Jordan algebras}\label{SectionHomJordan}
In this section, we consider, following \cite{Mak:HomAlternativeHomJord}, a generalization of Jordan algebra by twisting
the usual Jordan identity
$
(x\cdot y)\cdot x^2=x\cdot( y\cdot x^2).
$
We show that this
generalization fits with Hom-associative algebras. Also, we provide a procedure to
construct examples starting from an ordinary Jordan algebras.
\begin{definition}[\cite{Mak:HomAlternativeHomJord}]
A Hom-Jordan algebra is a triple $(V, \mu, \alpha)$
    consisting of a linear space $V$, a bilinear map $\mu: V\times V
    \rightarrow V$ which is commutative and a  homomorphism $\alpha: V \rightarrow V$
satisfying

\begin{equation}\label{HomJordanIdentity}
\mu(\alpha^2(x),\mu (y,\mu(x,x)))=\mu (\mu (\alpha (x),y),\alpha
(\mu(x,x)))
\end{equation}
where $\alpha^2=\alpha \circ \alpha$.
\end{definition}

\begin{remark}
Since the multiplication is commutative, one may write the identity
\eqref{HomJordanIdentity} as
\begin{equation}\label{HomJordanIdentity2}
\mu(\mu (y,\mu(x,x)),\alpha^2(x))=\mu (\mu (y,\alpha (x)),\alpha
(\mu(x,x))).
\end{equation}
\end{remark}
When the twisting map $\alpha$ is the identity map, we recover the classical
notion of Jordan algebra.

The identity \eqref{HomJordanIdentity} is motivated by the
following functor which associates to a Hom-associative algebra a
Hom-Jordan algebra by considering the plus Hom-algebra.

\begin{theorem}[\cite{Mak:HomAlternativeHomJord}]
Let $(V,m,\alpha)$ be a Hom-associative algebra. Then the  Hom-algebra $(V,\mu,\alpha)$,
where the multiplication $\mu$ is  defined for all $x,y \in V$ by
$$
\mu( x,y )=\frac{1}{2}(m (x,y)+m (y,x )).
$$
is a Hom-Jordan algebra.
\end{theorem}
\begin{proof}
The commutativity of $\mu$ is obvious. We   compute the difference
$$D=\mu(\alpha^2(x),\mu
(y,\mu(x,x)))-\mu (\mu
(\alpha(x),y),\alpha (\mu(x,x)))
$$
A straightforward computation gives
\begin{eqnarray*}
D&=& m(\alpha^2(x),m(y,m (x,x))) +m(m(y,m(x,x)),\alpha^2(x))
\\ \ && +
m(\alpha^2(x),m(m(x,x),y)) +m(m(m(x,x),y),\alpha^2(x))
\\ \ && -m(m(\alpha(x),y),\alpha(m(x,x))) -m(\alpha(m(x,x)),m(\alpha(x),y))
\\ \ &&
-m(m(y,\alpha(x)),\alpha(m(x,x)))
-m(\alpha(m(x,x)),m(y,\alpha(x))).
\end{eqnarray*}
We have by Hom-associativity
\begin{eqnarray*}
m(\alpha^2(x),m(y,m (x,x)))-m(m(\alpha(x),y),\alpha(m(x,x)))=0\\
m(m(m(x,x),y),\alpha^2(x))-m(\alpha(m(x,x)),m(y,\alpha(x)))=0.
\end{eqnarray*}
Therefore
\begin{eqnarray*}
D&=&
m(m(y,m(x,x)),\alpha^2(x))+
m(\alpha^2(x),m(m(x,x),y))
\\ \ &&
-m(\alpha(m(x,x)),m(\alpha(x),y))
-m(m(y,\alpha(x)),\alpha(m(x,x))).
\end{eqnarray*}
One may show that for any Hom-associative algebra we have
\begin{eqnarray*}
m(\alpha(m(x,x)),m(\alpha(x),y))
&= m(m(m(x,x),\alpha(x)),\alpha(y))
\\&=m(m(\alpha(x),m(x,x)),\alpha(y))
\\&=m(\alpha^2(x),m(m(x,x),y)),
\end{eqnarray*}
and similarly
\begin{equation*}
m(m(y,\alpha(x)),\alpha(m(x,x)))=m(m(y,m(x,x)),\alpha^2(x)).
\end{equation*}
Thus
\begin{eqnarray*}
D&=&
m(m(y,m(x,x)),\alpha^2(x))+
m(\alpha^2(x),m(m(x,x),y))
\\ \ &&
-m(\alpha^2(x),m(m(x,x),y))
-m(m(y,m(x,x)),\alpha^2(x))
\\ \ &=&0.
\end{eqnarray*}
\end{proof}
\begin{remark}
The definition of Hom-Jordan algebra seems to be non natural one
expects that the identity should be of the form
\begin{equation}\label{HomJordanIdentity2}
\mu(\alpha(x),\mu (y,\mu(x,x))))=\mu (\mu (x,y),\alpha (\mu(x,x)))
\end{equation}
or
\begin{equation}\label{HomJordanIdentity3}
\mu(\alpha(x),\mu (y,\mu(x,x))))=\mu (\mu (x,y),\mu(x,\alpha (x))).
\end{equation}
It turns out that these identities do not fit with the previous
proposition.

 Notice also that in general a Hom-alternative algebra doesn't lead to a
 Hom-Jordan algebra.
\end{remark}

The following theorem gives a procedure to construct Hom-Jordan algebras
using ordinary Jordan algebras and their algebra
endomorphisms.
\begin{theorem}[\cite{Mak:HomAlternativeHomJord}]\label{thmConstrHomJordan}
Let $(V,\mu)$ be a Jordan algebra and $\alpha : V\rightarrow V$ be
an
 algebra endomorphism. Then $(V,\mu_\alpha,\alpha)$,
where $\mu_\alpha=\alpha\circ\mu$,  is a  Hom-Jordan
algebra.

Moreover, suppose that  $(V',\mu')$ is another
Jordan algebra  and $\alpha ' : V'\rightarrow V'$ is an algebra
endomorphism. If $f:V\rightarrow V'$ is an algebras morphism that
satisfies $f\circ\alpha=\alpha'\circ f$ then
$$f:(V,\mu_\alpha,\alpha)\longrightarrow (V',\mu'_{\alpha '},\alpha ')
$$
is a morphism of  Hom-Jordan algebras.
\end{theorem}
%\begin{proof}
%We show that $(V,\mu_\alpha,\alpha)$ satisfies the Hom-Jordan
%identity (\eqref{HomJordanIdentity}) while $(V,\mu)$ satisfies the
%Jordan identity (\eqref{JordanIdentity}). Indeed
%\begin{align*}
%& \mu_\alpha(\alpha^2(x),\mu_\alpha
%(y,\mu_\alpha(x,x)))-\mu_\alpha (\mu_\alpha
%(\alpha(x),y),\alpha (\mu_\alpha(x,x)))
%\\ &  = \alpha\circ \mu(\alpha^2(x),\alpha\circ \mu
%(y,\alpha\circ \mu(x,x)))-\alpha\circ \mu (\alpha\circ \mu
%(\alpha(x),y),\alpha^2\circ \mu(x,x))
%\\ &  = \alpha^2( \mu(\alpha(x), \mu
%(y,\alpha\circ \mu(x,x)))- \mu ( \mu
%(\alpha(x),y),\alpha\circ \mu(x,x)))
%\\ &  = \alpha^2( \mu(\alpha(x), \mu
%(y,\mu(\alpha(x),\alpha(x))))- \mu ( \mu
%(\alpha(x),y), \mu(\alpha(x),\alpha(x))))
%\\ &  =0.
%\end{align*}
%
%the second assertion follows from
%$$f\circ \mu_\alpha=f\circ \alpha \circ \mu=\alpha'\circ
%f \circ \mu =\alpha'\circ \mu' \circ f  =\mu'_{\alpha'} \circ f.
%$$
%\end{proof}
\begin{remark}
We may give here similar observations as in the remark
(\ref{HomAltInducedByAlt})
concerning Hom-Jordan algebra  induced by an
ordinary Jordan algebra.
\end{remark}
We provide in the sequel an example of Hom-Jordan algebras.

\begin{example}
We consider Hom-Jordan algebras associated to  Hom-associative
 algebras described in example (\ref{example1ass}).
Let $\{e_1,e_2,e_3\}$  be a basis of a $3$-dimensional linear space
$V$ over $\K$. The following multiplication $\mu$ and linear map
$\alpha$ on $V$ define  Hom-Jordan algebras over $\K^3${\rm :}
$$
\begin{array}{ll}
\begin{array}{lll}
 \widetilde{\mu} ( e_1,e_1)&=& a\ e_1, \ \\
\widetilde{\mu} ( e_1,e_2)&=&\widetilde{\mu} ( e_2,e_1)=a\ e_2,\\
\widetilde{\mu} ( e_1,e_3)&=&\widetilde{\mu} ( e_3,x_1)=b\ e_3,\\
 \end{array}
 & \quad
 \begin{array}{lll}
\widetilde{\mu} ( e_2,e_2)&=& a\ e_2, \ \\
\widetilde{\mu} ( e_2, e_3)&=& \frac{1}{2}b\ e_3, \ \\
\widetilde{\mu} ( e_3,e_2)&=& \widetilde{\mu} ( e_3,e_3)=0,
  \end{array}
\end{array}
$$

$$  \alpha (e_1)= a\ e_1, \quad
 \alpha (e_2) =a\ e_2 , \quad
   \alpha (e_3)=b\ e_3
$$
where $a,b$ are parameters in $\K$.

It turns out that the multiplication of this Hom-Jordan algebra defines
 a Jordan algebra.

\end{example}

\begin{remark}
We may define the noncommutative Jordan algebras as triples
$(V,\mu,\alpha)$ satisfying the identity \eqref{HomJordanIdentity}
and the flexibility condition, which is a generalization of the
commutativity. Eventually, we may consider the Hom-flexibilty
defined by the identity $\mu(\alpha(x), \mu (y, x))= \mu (\mu (x,
y), \alpha (x)).$

\end{remark}
\begin{remark}
A $\mathbb{Z}_2$-garded version of Hom-alternative algebras and Hom-Jordan algebras
might be defined in a natural way.
\end{remark}

{\ding{167}}

\section*{Acknowledgments}
Dedicated to Amine Kaidi on his 60th birthday.

\end{document}